\documentclass[11pt]{amsart}

\usepackage[utf8]{inputenc}

\usepackage{amsthm}
\usepackage{amsmath, amssymb}
\usepackage{url}
\usepackage{color}
\usepackage{verbatim}

\usepackage{todonotes}

\usepackage{amsthm}
\usepackage{tkz-euclide}
\usetikzlibrary{patterns, intersections}

\newtheorem{theorem}{Theorem}[section]
\newtheorem{proposition}[theorem]{Proposition}
\newtheorem{lemma}[theorem]{Lemma}
\newtheorem{corollary}[theorem]{Corollary}

\theoremstyle{definition}
\newtheorem{definition}[theorem]{Definition}
\newtheorem{example}[theorem]{Example}

\newtheorem{remark}[theorem]{Remark}

\newtheorem{thmx}{Theorem}
\newtheorem{corx}[thmx]{Corollary}

\DeclareMathOperator{\C}{\mathcal{C}}
\DeclareMathOperator{\A}{\mathcal{A}}

\DeclareMathOperator{\Aut}{Aut}
\DeclareMathOperator{\Cay}{Cay}
\DeclareMathOperator{\Sym}{Sym}

\DeclareMathOperator{\link}{link}
\let\star\relax
\DeclareMathOperator{\star}{star}
\DeclareMathOperator{\id}{id}
\DeclareMathOperator{\stab}{stab}

\makeatletter
\renewcommand*{\@textcolor}[3]{%
  \protect\leavevmode
  \begingroup
    \color#1{#2}#3%
  \endgroup
}
\makeatother

\title[Automorphism groups of Cayley graphs of Coxeter groups]{On the non-discreteness of automorphism groups of Cayley graphs of Coxeter groups}

\author{Federico Berlai}
\address[Federico Berlai]{Department of Mathematics, University of the Basque Country UPV/EHU, Barrio Sarriena s/n, 48940 Leioa, Spain}
\email[Federico Berlai]{federico.berlai@gmail.com}

\author{Michal Ferov}
\address[Michal Ferov]{School of Information and Physical Sciences, University of Newcastle, University Drive, NSW 2308, Australia}
\email[Michal Ferov]{michal.ferov@gmail.com}
\date{\today}

\begin{document}

\maketitle
\begin{abstract}
In this work we characterise Cayley graphs of Coxeter groups with respect to the standard generating set that admit uncountable vertex stabilisers. As a corollary, we fully identify finitely generated Coxeter groups for which the automorphism group of their Cayley graph with respect to the standard generating set is not discrete when equipped with the permutation topology. As an application, we also provide new explicit constructions of vertex-transitive graphs of infinite degree that have locally compact automorphism groups.
\end{abstract}

\tableofcontents

\section{Introduction}
\label{sec:introduction}
An automorphism group of a connected graph $\Delta=(V\Delta,E\Delta)$ naturally admits a topology, namely the pointwise-convergence topology, also called the permutation topology, whose base of open sets at the identity is given by the family
\begin{displaymath}U_F:=\bigl\{ \varphi\in \Aut(\Delta)\mid \varphi(x)=x\ \forall x\in F\bigr\},\end{displaymath}
where \(F \subseteq V\Delta\) is finite. That is, basic open sets at the identity are finite intersections of vertex stabilisers. This topology is totally disconnected. Whenever the graph $\Delta$ is locally finite, it is also locally compact, so the group $\Aut(\Delta)$ is a totally disconnected locally compact (TDLC) topological group. For more details on permutation topologies, we refer the reader to Subsection \ref{subsec:permutation topologies}.

Being locally compact or not, this topology is second countable whenever the vertex set of the graph $\Delta$ is at most countable, so that in this case the group $\Aut(\Delta)$ is discrete if and only if it is countable. 

A natural problem, thus, is to find conditions on a countable graph $\Delta$ that assure this topology on $\Aut(\Delta)$ to be non-discrete (or, equivalently, that assure that the group $\Aut(\Delta)$ is not countable). This is a well-known problem that attracted attention from several branches of mathematics. For instance, we must mention the seminal paper \cite{Halin} of Halin that provides a graph theoretical characterisation of uncountability. Nevertheless, stronger forms of Halin's result were already known among (and published by) logicians (compare also \cite{ImSm}).

\smallskip
In this work we address the above-mentioned problem for a particular family of graphs, that is, Cayley graphs of Coxeter groups (we redirect the interested reader to Subsection~\ref{subsection_Coxeter} for the precise definitions and notation that will be used throughout this paper). In Theorem~\ref{thmA} we fully characterise, in terms of the defining graph $\Gamma$, those countably generated Coxeter groups $W_\Gamma$ for which $\Aut\bigl(\Cay(W_\Gamma,S)\bigr)$ has uncountable vertex stabilisers.

If $G$ is a group and $S$ is a generating set for $G$, then Cayley's theorem asserts that $G$ is a subgroup of $\Aut\bigl(\Cay(G,S)\bigr)$, and in particular it is the subgroup consisting of label-preserving automorphisms of the labelled graph $\Cay(G,S)$.

It is worth mentioning that recently results with an opposite flavour (compared to ours) appeared in the literature. Indeed, it has been proved in \cite[Theorem 7]{LdlS20} (see also \cite[Theorem 1.1]{LdlS}) that for any finitely generated group $G$ which is not abelian nor generalised dicyclic, if $S_0$ is a finite symmetric generating set for $G$ then $G\cong \Aut\bigl(\Cay(G,S)\bigr)$, where $S$ is the new generating set $(S_0\cup S_0^2\cup S_0^3)\setminus \{e\}$. Thus, the results from \cite{LdlS20,LdlS} mean that is it is always possible to find generating sets so that the Cayley graph has as few automorphisms as possible.

The impulse that led us to prove Theorem~\ref{thmA} came from~\cite{T}. Indeed, in~\cite{T} Taylor considers the similar question for Cayley graphs of finitely generated right-angled Artin groups with respect to the standard generating set, and proves that their automorphism group is countable if and only if the defining graph is complete, that is if and only if the corresponding right-angled Artin group is a finitely generated free abelian group.

\smallskip
To be able to state Theorem \ref{thmA} we will need to introduce now some notation.
Given an simplicial graph $\Gamma=(V\Gamma,E\Gamma)$, not necessarily finite, and a \emph{weight} function $m\colon V\Gamma\times V\Gamma\to \{1,2,\dots\}\cup\{+\infty\}$ satisfying
\begin{enumerate}
    \item[$c_1)$] $m(x,y)=m(y,x)$ for all $x,y\in V\Gamma$;
    \item[$c_2)$] $m(x,y)=1$ if and only if $x=y$;
    \item[$c_3)$] $m(x,y)=+\infty$ if and only if $\{x,y\}\notin E\Gamma$,
\end{enumerate}
we can consider the Coxeter group (notice, it will be finitely generated exactly when the graph $\Gamma$ is finite)
\begin{equation}\label{equation_Coxeter}
W_\Gamma:=\langle  V\Gamma\mid (xy)^{m(x,y)}=e\quad \forall\, x,y\in V\Gamma\rangle,
\end{equation}
where $(xy)^{+\infty}=e$ by definition means that there is no relation between the two generators $x$ and $y$. We call $\Gamma=(V\Gamma,E\Gamma, m)$ a weighted graph.
Notice that the Coxeter group $W_\Gamma$ is a right-angled Coxeter group whenever the weight function $m$ only takes values in $\{1,2,+\infty\}$. In this case, we usually forget about this function and plainly consider simplicial graphs $\Gamma=(V\Gamma,E\Gamma)$: for vertices $x\neq y$, we have that $m(x,y)=2$ if and only if $\{x,y\}\in E\Gamma$, and $m(x,y)=+\infty$ if and only if $\{x,y\}\notin E\Gamma$.

Let us denote by $\A_\Gamma$ the automorphism group of the Cayley graph of $W_\Gamma$ with respect to the standard generating set $V\Gamma$. Moreover, let $\Aut(\Gamma)$ denote the group of symmetries of $\Gamma$ that preserve weights (in particular, if we reduce our attention to right-angled Coxeter groups, then this is the automorphism group of the unweighted defining graph). Then we have the following statement.

\begin{thmx}\label{thmA}
\emph{Let $\Gamma$ be a countable weighted graph and let $\C_\Gamma$ be the Cayley graph of the Coxeter group $W_\Gamma$ with respect to the standard generating set $V\Gamma$. The automorphism group $\A_\Gamma$ has uncountable vertex stabilisers if and only if either $\Aut(\Gamma)$ is uncountable or there exists $x\in V\Gamma$ and a non-trivial automorphism $\alpha\in\Aut(\Gamma)$ such that $\alpha\restriction_{\star(x)}=\id_\Gamma$.
}
\end{thmx}

When the graph $\Gamma$ is finite, its group of weight-preserving symmetries $\Aut(\Gamma)$ will be finite as well, and in particular not uncountable. Notice that $\Gamma$ is finite if and only if
the Coxeter group $W_\Gamma$ is finitely generated, if and only if $\mathcal{C}_\Gamma$ is locally finite.

Thus, in that case (which is the case of interest for us), Theorem \ref{thmA} says that the TDLC group $\mathcal{A}_\Gamma$ has uncountable vertex stabilisers, and in particular it is itself uncountable, if and only if there exists a vertex $x$ in the finite graph $\Gamma$ that has some specific property.

Thus, we immediately deduce the following result for TDLC groups:
\begin{corx}\label{corA}
\emph{
Let $\Gamma$ be a finite weighted graph and let $\C_\Gamma$ be the Cayley graph of the Coxeter group $W_\Gamma$ with respect to the standard generating set $V\Gamma$. The group $\A_\Gamma$ is a non-discrete TDLC group if and only if there exists $x\in V\Gamma$ and a non-trivial automorphism $\alpha\in\Aut(\Gamma)$ such that $\alpha\restriction_{\star(x)}=\id_\Gamma$.
}
\end{corx}

From Theorem \ref{thmA} we immediately notice that the situation for (right-angled) Coxeter groups is diametrically opposed to the one of right-angled Artin groups: there exist graphs $\Gamma$ such that $\A_\Gamma$ is countable, and $\Gamma$ admits an induced subgraph $\Delta$ for which $\A_\Delta$ is uncountable. As an example, we can consider the graphs
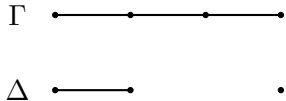
\begin{equation*}
\begin{tikzpicture}
\draw[thick] 
(-1,0) --(0,0) --(1,0) --(2,0)

(-1,-1) --(0,-1);
\draw[fill] 
(-1,0) circle [radius=0.035] node[xshift=-.5cm]{$\Gamma$}
(0,0) circle [radius=0.035]
(1,0) circle [radius=0.035]
(2,0) circle [radius=0.035]

(-1,-1) circle [radius=0.035] node[xshift=-.5cm]{$\Delta$}
(0,-1) circle [radius=0.035]
(2,-1) circle [radius=0.035];
\end{tikzpicture}
\end{equation*}
Further, we also note that the Coxeter groups $W_\Gamma$ and $W_\Delta$ are quasi-isometric. This means that the property of having a non-discrete group of automorphism of the Cayley graph (with respect to the standard generating set) is not a quasi-isometry invariant within the class of Coxeter groups.

The paper is organised as follows. In Section \ref{section2} we introduce standard notation and definitions for graphs, Coxeter groups, permutation topologies, and provide the first results concerning the geometry of Cayley graphs of Coxeter groups, and their automorphisms.
In Section \ref{section3} and Section \ref{section4} we prove Theorem \ref{thmA}, each section concerned with one implication. We develop the notion of \emph{good separating set} in Definition \ref{goodseparatingset}. By Lemma \ref{char_good_sep_set}, the presence of a good separating set is equivalent to the condition appearing in Theorem \ref{thmA}. We conclude with Section \ref{section6}, where we show that the tools developed for the project can be also used to give explicit constructions of vertex transitive graphs that are not locally finite and yet still have locally-compact group of automorphisms.

\subsection{Previously known results}
When working on this article, the authors were not aware of any previously published results on (non)discreteness of Automorphism groups of Cayley graphs of Coxeter groups. Only when this project was finished, we found out about \cite{clarke} and \cite{haglund}.

The following is is a restatement of \cite[Theorem 5.12]{haglund}
\begin{theorem}
    For a finite graph $\Gamma$, if there exists $x\in V\Gamma$ and a non-trivial automorphism $\alpha\in\Aut(\Gamma)$ such that $\alpha\restriction_{\star(x)}=\id_\Gamma$ then the automorphism group $\Aut(\Cay(W_\Gamma, V\Gamma))$ is discrete, and is the semidirect product of $W_\Gamma$ and $\Aut(\Gamma)$. If not, then if the group $W_\Gamma$ is word-hyperbolic in the sense of Gromov, the automorphism group $\Aut(\Cay(W_\Gamma, V\Gamma))$ is not discrete.
\end{theorem}
Compared to the above, Corollary \ref{corA} is more general, as it makes no assumptions on the geometry of the group $W_\Gamma$.

The results presented in Graham Clarke's master thesis, in particular \cite[Theorem 4.2 and Proposition 4.31]{clarke}, are equivalent to Corollary \ref{corA}. However, we feel that our methods are more general and that our proofs are less technical. In particular, we intend to use the machinery of configurations of local actions in further projects.

\subsection*{Acknowledgements} 
When this work started, Federico Berlai was visiting the Department of Mathematics of the University of the Basque Country UPV/EHU, whose hospitality he greatly appreciates. Federico Berlai was supported by the Austrian Science Foundation FWF, grant no.~J4194, and gratefully acknowledges support from the Basque Government Grant IT974-16, from the Spanish Government Grants MTM2017-86802-P, partly with FEDER funds, and PID2020-117281GB-I00.

Majority of the work on this project was done while Michal Ferov was stuck in Europe for 21 months due to Australian policy of closed borders during the global pandemic of covid-19. During that time, and to the current day still, Michal Ferov was supported by the Australian Research Council Laureate Fellowship FL170100032 of professor George Willis.

The authors would like to thank Bruno Duchesne, Robert Kropholler, and Phillip Wesolek for sharing with us an old unpublished preprint on related topics. We thank Colin Reid for useful discussions and the suggestion to use our methods to the setting of infinitely generated Coxeter groups. We thank Anne Thomas for making us aware of a honours thesis written by her student Graham Clarke \cite{clarke} and the work of Haglund and Paulin \cite{haglund}.

\section{Preliminaries}\label{section2}
In this section we fix the notation and collect known and preliminary results that will be of use later.
The neutral element of a group $G$ is denoted by $e_G$, or simply by $e$ if the group is clear from the context. In case when $G = \Aut(\Delta)$, where $\Delta$ is some graph, we will also use $\id_{\Delta}$ to denote the neutral element in $G$, signifying that it is the identity map $\id_\Delta\colon \Delta\to \Delta$. Again, if the graph $\Delta$ is clear from the context we will omit the subscript. 

If $G$ is a group acting of a set $X$, then for a subset $S \subseteq X$ we will use $\stab_G(S)$ to denote the \emph{point-wise stabiliser} of $S$, i.e.
\begin{displaymath}
    \stab_G(S) = \{g \in G \mid g\cdot s = s \mbox{ for all $s \in S$}\}.
\end{displaymath}
We will use the following conventions with respect to the naming of elements. Given a Coxeter group $W_\Gamma$, we will denote by $x,y,\dots$ elements from its (fixed) generating set $V\Gamma$, by $u,v,w,\dots$ the vertices in its Cayley graph which we will identify with elements of the group $W_\Gamma$, and by $U,V,W$ we will usually denote words over the alphabet $V\Gamma$, i.e. elements of the free monoid $(V\Gamma)^*$ . The letter $\sigma$ will usually denote an element from the automorphism group of the graph $\Gamma$ associated to the Coxeter group, whereas $\alpha,\beta,\dots$ will denote elements of the automorphism group of the Cayley graph of $W_\Gamma$.

\subsection{Graphs}
In this work we will focus at the same time on Cayley graphs, and on simplicial graphs with weights associated to their edges. A Cayley graph of a group $G$ with respect to a generating set $S$ is an $S$-labelled oriented graph $\Cay(G,S)$, where the vertices of $\Cay(G,S)$ are identified with elements of $G$ and an ordered pair $(v,w)\in G\times G$ corresponds to an oriented edge labelled by $s \in S$ if and only if $w=vs$. In general, a Cayley graph may contain multiple edges and loops. However, since in this note we will only consider Cayley graphs of Coxeter groups generated by distinct non-trivial involutions, we might assume that $\Cay(G,S)$ is a $S$-labelled simplicial graph, i.e. an undirected graph without loops and multiple edges. When talking about automorphisms of Cayley graphs, we will consider bijections of the vertex set $G$ that preserve adjacency between vertices, but not necessarily the labels.

In the rest of this paper, by a graph we mean a simplicial graph, that is a graph $\Gamma=(V\Gamma,E\Gamma)$ where $V\Gamma$ is a set and $E\Gamma \subseteq \binom{V\Gamma}{2}$ is the set of edges, that is unoriented pair of vertices, such that there are no loops (edges of the form $\{x,x\}$ for $x\in V\Gamma$) nor multiple edges between the same pair of vertices. Two vertices $x,y\in V\Gamma$ are adjacent if $\{x,y\}\in E\Gamma$. A map $\sigma\colon V\Gamma\to V\Gamma$ is an automorphism of the graph $\Gamma$ if it is a bijection and $\{v,w\}\in E\Gamma$ if and only if $\{\sigma(v),\sigma(w)\}\in E\Gamma$.

By a subgraph we will always mean an induced subgraph, i.e. a subgraph whose edge set is fully determined by its vertices. Given two simplicial graphs $\Gamma = (V\Gamma, E\Gamma)$ and $\Delta = (V\Delta, E\Delta)$, we say that $\Gamma$ is an induced subgraph of $\Delta$ if $V\Gamma \subseteq V\Delta$ and $V\Gamma = \binom{V\Gamma}{2} \cup E\Delta$.

Let $m\colon V\Gamma\times V\Gamma\to \{1,2,\dots\}\cup\{+\infty\}$ be a function satisfying the following properties:
\begin{enumerate}
    \item[$c_1)$] $m(x,y)=m(y,x)$ for all $x,y\in V\Gamma$;
    \item[$c_2)$] $m(x,y)=1$ if and only if $x=y$;
    \item[$c_3)$] $m(x,y)=+\infty$ if and only if $\{x,y\}\notin E\Gamma$.
\end{enumerate}
We call $(\Gamma,m)=(V\Gamma,E\Gamma,m)$ a \emph{weighted graph} and $m$ a \emph{weight} for $\Gamma$. When the context is clear, we will hide $m$ from the notation and say that $\Gamma$ is a weighted graph.

A map $\sigma\colon V\Gamma\to V\Gamma$ is an automorphism of the weighted graph $\Gamma$ if it is a bijection such that $m(x,y)=m\bigl(\sigma(x),\sigma(y)\bigr)$ for all $x,y\in V\Gamma$. In particular, if $\sigma$ is an automorphism of $\Gamma$ then $\{x,y\}\in E\Gamma$ if and only if $\{\sigma(x),\sigma(y)\}\in E\Gamma$.
The group of automorphisms of a weighted graph $\Gamma$ is denoted by $\Aut(\Gamma)$.

If $m$ is a weight such that $m(V\Gamma\times V\Gamma)\subseteq \{1,2,+\infty\}$ then, by definition of weight, $m(x,y)=1$ if and only if $x=y$, and $m(x,y)=2$ if and only if $\{x,y\}\in E\Gamma$, and $m(x,y)=+\infty$ if and only if $\{x,y\}\notin E\Gamma$. In this case, a bijection of $V\Gamma$ is an automorphism of the weighted graph $(\Gamma,m)$ if and only if it is an automorphism of the unweighted graph $\Gamma$.

A proper subset of vertices $S\subsetneq V\Gamma$ is a \emph{separating set} if the induced subgraph spanned by $V\Gamma\setminus S$ is disconnected. By definition, if $\Gamma$ is already disconnected, the empty set is a separating set for the graph $\Gamma$.

\subsection{Permutation topologies}
\label{subsec:permutation topologies}
In this subsection we recall the notion of permutation topology and some basic facts. Readers familiar with the terminology may wish to skip this subsection.

Let  $X$ be a set and suppose that $G$ is a group acting on $X$. There is a group topology on $G$ naturally arising from its action on $X$, namely the \emph{permutation topology} or sometimes called the \emph{point-wise convergence topology}. The base of neighbourhoods of the identity consists of the collection of sets of the form $\stab_G(F)$, where $F \subset X$ is finite.

It then follows that for an element $g \in G$ its base of neighbourhoods consists of all the sets of the form
\begin{displaymath}
   \mathcal{B}(g, F) = \{f \in G \mid f\restriction_{F} = g\restriction_{F}\}.
\end{displaymath}
where, again, the subset $F \subset X$ is finite.

The following statement then follows immediately from the definition.
\begin{lemma}
    A subgroup $H \leq G$ is open in the permutation topology on $G$ arising from an action on a set $X$ if and only if it contains the pointwise stabiliser of some finite set $F \subset X$.
\end{lemma}
Quite clearly, the kernel of the action is contained in every neighbourhood of the identity, so it is the smallest (with respect to inclusion) open subset containing the identity. In fact, the kernel of the action is the connected component of the identity - it then follows that the permutation topology is totally disconnected if an only if the action is faithful.

Further, if the action is faithful, then for each tuple of distinct elements $f,g \in G$ there exists $x \in X$ such that $f\cdot x \neq g \cdot x$. In particular, this means $\mathcal{B}(f, \{x\})$ and $\mathcal{B}(g,\{x\})$ are disjoint open neighbourhoods of $f$ and $g$ respectively. Conversely, following the definition of permutation topology, we see that if the permutation topology arising from the action on $X$ is not Hausdorff, then there exists $g \in X$ contained in every neighbourhood of the identity, therefore $g \cdot x = x$ for all $x \in X$.

The following lemma sums up the following two observations.
\begin{lemma}
    The permutation topology on $G$ is Hausdorff if and only if the action of $G$ on $X$ is faithful. Furthermore, $G$ is totally disconnected if and only if the action is faithful.
\end{lemma}

Within the scope of this paper, $X$ will be a vertex transitive simplicial graph, with $X = (VX, EX)$, where $VX$ denotes the set of vertices of $X$ and $EX \subseteq \binom{VX}{2}$ denotes the set of edges of $X$, and $G \leq \Aut(X)$. In particular, this means that $G$ is totally disconnected. In the case when $X$ is locally finite, i.e. when each vertex is adjacent to only finitely many vertices, it can be easily seen that $\Aut(X)$ is locally compact.

\subsection{Coxeter groups}\label{subsection_Coxeter}
The \emph{right-angled Coxeter group} $W_\Gamma$ (also abbreviated by RACG) associated to a graph $\Gamma=(V\Gamma,E\Gamma)$ without weights is the group given by the presentation
\begin{equation*}
    W_\Gamma:=\langle V\Gamma\mid x^2=e\ \forall x\in V\Gamma,\, xy=yx\ \forall \{x,y\}\in E\Gamma\rangle.
\end{equation*}
More generally, 
a \emph{Coxeter group} is a group (not necessarily finitely generated) given by the presentation of the form 
\begin{equation}\label{equation_Coxetergroup}
       W:= \left\langle x_1, \dots, x_n, \dots {} \mid \left( x_i x_j \right)^{m_{i,j}}=e\ \forall i,j \right\rangle,
\end{equation}
where the exponents $m_{i,j}\in \{1,2,\dots\}\cup \{+\infty\}$ satisfy:
\begin{enumerate}
    \item $m_{i,j}=m_{j,i}$ for all $i,j$;
    \item $m_{i,j}=1$ if and only if $i=j$, that is, generators are involutions;
    \item if $m(x,y)=+\infty$ then, by definition, there is no relation between $x_i$ and $x_j$.
\end{enumerate}
In the literature the set $S=\{x_1,\dots,x_n\}$, when finite, is called a \emph{Coxeter generating set}, and the pair $(W,S)$ is called a \emph{Coxeter system}, although in this paper we will not enforce this terminology.

To each Coxeter system (we maintain the notation of Equation \eqref{equation_Coxetergroup}) we can associate a weighted graph $\Gamma$ in the following way. The vertex set $V\Gamma$ is defined to be the set $\{x_1,\dots,x_n,\dots\}$. Two different vertices $x_i$ and $x_j$ are joined by an edge in $\Gamma$ if and only if $m_{i,j}<+\infty$, and the weight map is defined as $m(x_i,x_j):=m_{i,j}$.

Therefore, any Coxeter group is of the form
\begin{equation}\label{equation_defining_graph}
    W_\Gamma:=\langle V\Gamma\mid (xy)^{m(x,y)}=e\,\,\forall x,y\in V\Gamma \rangle,
\end{equation}
where $(\Gamma,m)$ is a weighted graph. Again, $(xy)^{+\infty}=e$ means that there is no relation imposed between the generators $x$ and $y$.

Notice that a Coxeter group is a right-angled Coxeter group if and only if its weight function satisfies $m( V\Gamma\times V\Gamma)\subseteq \{1,2,+\infty\}$. That is, we can think of right-angled Coxeter groups as Coxeter groups in which all edges have weight equal to two.

\smallskip
Given a Coxeter group $W_\Gamma$, we denote by $\C_{\Gamma}$ the Cayley graph of $W_\Gamma$ with respect to the generating set $V\Gamma$, that is the graph whose vertex set is $W_\Gamma$, and such that two vertices $v,w\in W_\Gamma$ are joined by an edge if and only if there exists $x\in V\Gamma$ satisfying $v=wx$. In this case, the edge $\{v,w\}$ is \emph{labelled} by the letter $x$, that is $\C_\Gamma$ is labelled by the set $V\Gamma$.

Notice that the Cayley graph is regular, connected, and each vertex has valency $\lvert V\Gamma\rvert$, each incident edge being labelled by exactly one element of $V\Gamma$. In particular, this Cayley graph will be locally finite if and only if the generating set $V\Gamma$ is finite. Given a vertex $v$ in the Cayley graph and a natural number $n$, we denote by $B(v,n)$ the ball of radius $n$ in $\C_\Gamma$ around the vertex $v$.

We denote by $\A_\Gamma$ the group of automorphisms $\Aut(\C_\Gamma)$. When the graph $\Gamma$ is clear from the context, we will often drop the subscript $\Gamma$ to ease notation. We are not imposing that elements of $\A_\Gamma$ preserve labels.

\subsection{Reduced forms, parabolic subgroups}

\medskip
Suppose that $\Gamma$ is an unweighted graph (or, equivalently, a graph with a weight function whose images are contained in $\{1,2,+\infty\}$).
Given  a pair of vertices $x,y\in V\Gamma$, that is a pair of canonical generators of the right-angled Coxeter group $W_\Gamma$, we have that $x$ and $y$ are joined by an edge in $\Gamma$ if and only if, for any vertex $v$ in the Cayley graph~$\C_\Gamma$, we find an embedded cycle of length four in $\C_\Gamma$ that has $v$ as a vertex and where opposite sides of the square are labelled by $x$ and $y$ respectively (compare with Figure \ref{fig1}). We call such embedded cycle a \emph{commuting square} at $v$ with labels $x$ and $y$.
\begin{figure}[ht!]
\begin{tikzpicture}
\draw (0,0) circle (1.5cm);

\filldraw (0,0) circle (0.035) node[yshift=-.25cm] {$v$}
(1.5,0) circle (0.035) node[yshift=-.25cm ,xshift=.25cm] {$vx$};

\begin{scope}[rotate=45]
\filldraw (1.5,0) circle (0.035) node[yshift=.25 cm] {$vy$};

\begin{scope}[shift={(1.5,0)},rotate=45]
\filldraw (0,-1.5) circle (0.035) node[yshift=.25cm] {$vxy$};
\end{scope}
\end{scope}

\draw[thick,red] (0,0) -- (1.5,0);

\begin{scope}[rotate=45]
\draw[thick,blue] (0,0) -- (1.5,0);
\end{scope}

\begin{scope}[shift={(1.5,0)},rotate=45]
\draw[thick,blue] (0,0) -- (1.5,0);
\end{scope}

\begin{scope}[shift={({1.5*cos(45)},{(1.5*cos(45))})}]
\draw[thick,red] (0,0) -- (1.5,0);
\end{scope}
\end{tikzpicture}
\caption{Commuting square at $v$.}
\label{fig1}
\end{figure}
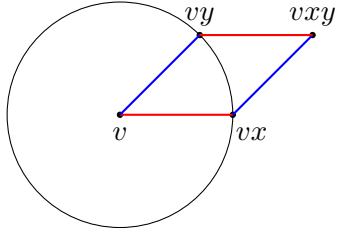

\noindent
The same observation extends to weighted graphs, that is to Cayley graphs of Coxeter groups: given any vertex $v$ in $\C_\Gamma$, two generators $x$ and $y$ will alternatively label the edges of an embedded cycle of even length $2l$ (one of whose vertices is $v$) if and only if $\{x,y\}\in E\Gamma$ and $m(x,y)=l$. We call this embedded cycle the \emph{relation cycle} of $x$ and $y$ at the vertex $v$. That is, commuting squares are exactly the relation cycles of length four.

\medskip
To prove Lemma \ref{lemma21} we need some notions for Coxeter groups. Let $\Gamma=(V\Gamma,E\Gamma,m)$ be a weighted simplicial graph, and consider the free monoid $V\Gamma^*$, that is the set of finite sequences of $V\Gamma$-elements, which we call \emph{words}. 
The elements $x_1,\dots,x_r$ in a word $W = (x_1,\dots,x_r)\in V\Gamma^*$ are called the \emph{syllables} of the word, and we say that $r$ is the \emph{length} of the word, writing $r = |W|$. We will use $\epsilon$ to denote the empty word.

Any such finite sequence $(x_1,\dots,x_r)$, or word, defines a group element via the evaluation map $\omega: V\Gamma^*\to W_\Gamma$ defined by
\begin{equation*}
\omega\bigl(x_1,\dots,x_r\bigr):=x_1\dots x_r\in W_\Gamma,
\end{equation*}
where $\omega(\epsilon):=e\in W_\Gamma$. We say that the word $(x_1,\dots,x_r)$ represents the group element $\omega\bigl(x_1,\dots,x_r\bigr)$.
Notice that different words may represent the same group element. 

Let us consider the following \emph{elementary simplifications} on $V\Gamma^*$:
\begin{itemize}
    \item[$s_1)$] if $y\in V\Gamma$ and two consecutive syllables of $(x_1,\dots,x_r)$ are equal to $y$, then remove these two syllables to obtain a sequence of length $r-2$;
    \item[$s_2)$] if $x,y\in V\Gamma$ and $m(x,y)$ consecutive syllables of $(x_1,\dots,x_r)$ are alternatively equal to $x$ and $y$, then in these $m(x,y)$ syllables replace all occurrences of $x$ with $y$ and vice versa.
\end{itemize}
If a word $w_2$ can be obtained from a word $w_1$ with finitely many elementary simplifications, then $\omega(w_1)=\omega(w_2)$.
We call the second kind of simplification a \emph{braid relation}. Rephrasing it, a braid relation is a substitution of the form
\begin{equation*}
\bigl(x_1,\dots, \underbrace{x,y,x,\dots}_{\text{length }m(x,y)},\dots,x_r\bigr)\quad \longmapsto\quad \bigl(x_1,\dots, \underbrace{y,x,y,\dots}_{\text{length }m(x,y)},\dots,x_r\bigr)
\end{equation*}
among finite sequences in $V\Gamma^*$. In particular, it must be that $m(x,y)\neq +\infty$ for a braid relation to occur. Moreover, as they are tautologically true, we will never consider braid relations for $x=y$, that is for $m(x,y)=1$, as in this case the braid relation does not modify the sequence.
Braid relations do not affect length of words, whereas a simplification of the first kind reduces the length by two. 

In the context of right-angled Coxeter groups, a braid relation is called \emph{syllable swapping}, because in that case we replace the subword $(x,y)$ with the word $(y,x)$.

The following fundamental result (compare \cite[Theorem~3.3.1]{BB}) characterises words of minimal length among words representing a given element $g \in W_\Gamma$.
\begin{theorem}\label{Tits_theorem}
Let $W_\Gamma$ be a Coxeter group and $w_1,w_2\in V\Gamma^*$ be two words of minimal length representing an element $g\in W_\Gamma$. Then $w_1$ can be obtained from $w_2$ applying a finite number of braid relations, and vice versa.
\end{theorem}

If $g\in W_\Gamma$ is an element of the Coxeter group $W_\Gamma$, with $\lVert g\rVert$ we denote the length of any word $w\in V\Gamma^*$ of minimal length representing the element, i.e.
\begin{displaymath}
    \|g\| = \min_{W \in V\Gamma^*} \{|W| \mid \omega(W) = g\}
\end{displaymath}

An automorphism $\alpha\in \A_\Gamma$ and a vertex $v$ in the Cayley graph $\C_\Gamma$ uniquely define an automorphism $\sigma(\alpha,v)\in\Aut(\Gamma)$, as follows. As $\alpha$ is an automorphism, it induces a bijection between the edges incident to $v$ and the edges incident to $\alpha(v)$, and therefore a bijection $\sigma(\alpha,v)\colon V\Gamma\to V\Gamma$ between their labels. We call $\sigma(\alpha, v)$ the \emph{local action of $\alpha$ at $v$} and we claim that this map is an automorphism of the weighted graph~$\Gamma$, that is, it preserves the weights $m(x,y)$ for all $x,y\in V\Gamma$. 

To demonstrate this, let $\{x_1,x_2\}\in E\Gamma$ be an edge, and let $w_1,w_2$ be vertices in the Cayley graph $\C_\Gamma$ adjacent to $v$ such that $\{v,w_1\}$ is labelled by $x_1$ and $\{v,w_2\}$ is labelled by $x_2$.
In the Cayley graph $\C_\Gamma$ we then see two relation cycles of length $2m(x_1,x_2)$ at the vertices $v$ and $\alpha(v)$, labelled by $x_1,x_2$ and by $\alpha(x_1),\alpha(x_2)$ respectively ($m(x_1,x_2)=2$ in Figure \ref{fig2}, and different colours correspond to different labels: blue for $x_1$, red for $x_2$, cyan for the label of $\{\alpha(v),\alpha(w_1)\}$, magenta for the label of $\{\alpha(v),\alpha(w_2)\}$).
\begin{figure}[ht]
\begin{tikzpicture}
\draw (0,0) circle (1.5cm);

\filldraw (0,0) circle (0.035) node[yshift=-.25cm] {$v$}
(1.5,0) circle (0.035) node[yshift=-.25cm ,xshift=.25cm] {$w_2$};

\begin{scope}[rotate=45]
\filldraw (1.5,0) circle (0.035) node[yshift=.25 cm] {$w_1$};

\begin{scope}[shift={(1.5,0)},rotate=45]
\filldraw (0,-1.5) circle (0.035) node[yshift=.25cm] {$w$};
\end{scope}
\end{scope}

\draw[thick,red] (0,0) -- (1.5,0);

\begin{scope}[rotate=45]
\draw[thick,blue] (0,0) -- (1.5,0);
\end{scope}

\begin{scope}[shift={(1.5,0)},rotate=45]
\draw[thick,blue] (0,0) -- (1.5,0);
\end{scope}

\begin{scope}[shift={({1.5*cos(45)},{(1.5*cos(45))})}]
\draw[thick,red] (0,0) -- (1.5,0);
\end{scope}
\end{tikzpicture}
\qquad 
\begin{tikzpicture}
\begin{scope}[rotate=-35]
\draw (0,0) circle (1.5cm);

\filldraw (0,0) circle (0.035) node[yshift=.25cm] {$\alpha(v)$}
(1.5,0) circle (0.035) node[yshift=-.25cm ,xshift=.25cm] {$\alpha(w_2)$};

\begin{scope}[rotate=45]
\filldraw (1.5,0) circle (0.035) node[yshift=.25 cm] {$\alpha(w_1)$};

\begin{scope}[shift={(1.5,0)},rotate=45]
\filldraw (0,-1.5) circle (0.035) node[yshift=-.25cm] {$\alpha(w)$};
\end{scope}
\end{scope}

\draw[thick,magenta] (0,0) -- (1.5,0);

\begin{scope}[rotate=45]
\draw[thick,cyan] (0,0) -- (1.5,0);
\end{scope}

\begin{scope}[shift={(1.5,0)},rotate=45]
\draw[thick,cyan] (0,0) -- (1.5,0);
\end{scope}

\begin{scope}[shift={({1.5*cos(45)},{(1.5*cos(45))})}]
\draw[thick,magenta] (0,0) -- (1.5,0);
\end{scope}
\end{scope}
\end{tikzpicture}
\caption{Graph automorphisms preserve the length of embedded cycles.}
\label{fig2}
\end{figure}
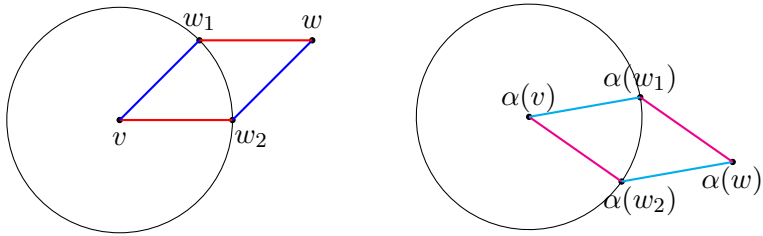

\noindent
Because $\alpha$ is an automorphism, the relation cycle at $v$ with labels $x_1$ and $x_2$ must be mapped to the relation cycle at $\alpha(v)$ with labels $\sigma(\alpha,v)(x_1)$ and $\sigma(\alpha,v)(x_2)$. Therefore $\sigma(\alpha,v)(x_1)$ and $\sigma(\alpha,v)(x_2)$ are joined by an edge in~$\Gamma$ and $m(x_1,x_2)=m\bigl(\sigma(\alpha,v)(x_1),\sigma(\alpha,v)(x_2)\bigr)$, that is the bijection $\sigma(\alpha,v)$ preserves weights. Therefore $\sigma(\alpha,v)\in \Aut(\Gamma)$. This means that we can think of $\sigma$ as a well defined map
\begin{displaymath}
    \sigma \colon \A_\Gamma \times W_\Gamma \to \Aut(\Gamma)
\end{displaymath}
and, similarly, for a given $\alpha \in \A \Gamma$ the map $\sigma(\alpha, -) \colon W_\Gamma \to \Aut(\Gamma)$ is well-defined.

\begin{definition}\label{definition_translation}
Let $\Gamma$ be a weighted graph and consider the Cayley graph~$\C_\Gamma$ of the associated Coxeter group. We say that $\alpha\in \A_\Gamma$ is 
\begin{enumerate}
    \item a \emph{translation} if $\sigma(\alpha, v) = \id_\Gamma$ for all $v \in W_{\Gamma}$, that is for all edges $\{v,w\}$ in~$\C_\Gamma$ the edges $\{v,w\}$ and $\{\alpha(v),\alpha(w)\}$ have the same label;
    \item an \emph{almost translation} if $\sigma(\alpha,v)=\sigma(\alpha,w)$ for all vertices $v,w$ in $\C_\Gamma$.
\end{enumerate}
\end{definition}
Equivalently, one can say that $\alpha \in \A_\Gamma$ is a translation if and only if $\sigma(\alpha, v) = \id_\Gamma$ for all $v \in W_\Gamma$, and that $\alpha$ is an almost translation if and only if the map $\sigma(\alpha, -)$ is constant.

We call these automorphisms translations because they correspond to left multiplications by $W_\Gamma$-elements, as recorded in Lemma \ref{label_preserving_autos_fg}.

Given $w\in W_\Gamma$, let us define $L_w\colon W_\Gamma\to W_\Gamma$ as $L_w(v):=wv$. As the Cayley graph $\C_\Gamma$ is a right-Cayley graph, $L_w$ is an element of the automorphism group of $\C_\Gamma$, that is $L_w\in \A_\Gamma$. In the following lemma we record that $w\mapsto L_w$ is an injective homomorphism whose image is the subgroup generated by translations of $\C_\Gamma$. More generally, this lemma works for any group, not just for Coxeter groups, and any generating set.

\begin{lemma}\label{label_preserving_autos_fg_general}
Let $G$ be a group, let $S$ be a generating set for $G$, and consider the (right) Cayley graph $\Cay(G,S)$. 
The map $w\mapsto L_w$ is an injective homomorphism, and therefore $G$ embeds as a subgroup of $\Aut\bigl(\Cay(G,S)\bigr)$.

Moreover, if two translations $\alpha$ and $\beta$ in $\Cay(G,S)$ are such that $\alpha(v)=\beta(v)$ for some vertex $v\in \Cay(G,S)$, then $\alpha=\beta$.
In particular, the set of all translations in $\Cay(G,S)$ is a subgroup isomorphic to $G$.
\end{lemma}
\begin{proof}
The map $\iota\colon G\to \Aut\bigl(\Cay(G,S)\bigr)$ defined by $w\mapsto L_w$ is an injective homomorphism by Cayley's theorem, and $L_w$ is a label-preserving for any $w\in G$.

Let $v\in \Cay(G,S)$ be such that $\alpha(v)=\beta(v)$. Thus, as $\alpha$ and $\beta$ are label-preserving, we deduce that they agree on the ball of radius one around~$v$, that is $\alpha(u)=\beta(u)$ for all $u\in B(v,1)$. As the Cayley graph is a connected graph, repeating this argument shows that $\alpha=\beta$.

To prove that the image of the map $\iota$ is the set of translations in the group $\Aut\bigl(\Cay(G,S)\bigr)$, that is the set of label-preserving automorphisms, consider a translation $\alpha\in \Aut\bigl(\Cay(G,S)\bigr)$. 
 Then $\alpha(e)=L_{\alpha(e)}(e)$, and thus by what proved so far $\alpha=L_{\alpha(e)}$.
\end{proof}
Thus, the subgroup of label-preserving automorphisms of a Cayley graph $\Cay(G,S)$ does not depend on the choice of generating set $S$.
Specialising Lemma~\ref{label_preserving_autos_fg_general} to our case of interest, that is to Coxeter groups, we obtain:
\begin{lemma}\label{label_preserving_autos_fg}
Let $\alpha$ and $\beta$ be two translations in $\A_\Gamma$ whose image coincide on a vertex $v\in \C_\Gamma$. Then $\alpha=\beta$. In particular, the set of all translations in $\A_\Gamma$ is a subgroup isomorphic to $W_\Gamma$.
\end{lemma}

In view of Lemma \ref{label_preserving_autos_fg},
with a slight abuse of notation we denote the set of translations by $W_\Gamma$.
Thus, the expressions $y\in W_\Gamma$ and $L_y\in W_\Gamma$ make both sense. The first will mean that we are considering the element $y$ is in the group $W_\Gamma$, whereas the second will mean that we are considering the translation $L_y$ in $\A_\Gamma$.

\begin{lemma}\label{autGamma_inclusion}
The subgroup of $\A_\Gamma$ generated by almost translations that fix the identity element is isomorphic to $\Aut(\Gamma)$.
\end{lemma}
\begin{proof}
Given $\sigma\in \Aut(\Gamma)$ we construct an almost translation $\alpha$ that fixes the identity vertex and such that $\sigma(\alpha,v)=\sigma$ for any vertex $v\in \C_\Gamma$, and we will then prove that the map $\sigma\mapsto \alpha$ is an isomorphism.

The construction is clear if $\sigma=\id_\Gamma$ is the identity automorphism, because we choose $\alpha$ to be the identity automorphism of $\C_\Gamma$. Hence, suppose $\sigma\neq \id_\Gamma$.
Incident to any vertex there are exactly $\lvert V\Gamma\rvert$ edges, each one being labelled by a different vertex of $\Gamma$.

Let us define the automorphism $\alpha$. 
First of all, we impose that $\alpha(e)=e$. If $w\in B(e,1)\setminus\{e\}$ is a vertex such that the edge $\{e,w\}$ is labelled by the letter $x$ and $\sigma(x)=y\in V\Gamma$, then we define $\alpha(w):= \tilde w$, where $\{e,\tilde w\}$ is the (unique) edge in $\C_\Gamma$ incident to $e$ that is labelled by $y$. Notice that $\tilde w\in B(e,1)\setminus \{e\}$.
This just means that $\sigma(\alpha,e)=\sigma$.

To complete the definition of $\alpha$ we proceed as follows, inductively. 
Let $w\in \C_\Gamma$ be a vertex at distance $l\geqslant 1$ from $e$, consider a path $\gamma$ with vertices $(e,w_1,w_2,\dots,w_l=w)$ of length $l$ that connects $e$ to $w$ in $\C_\Gamma$, and suppose that $\alpha$ has already been defined on $B(e,l-1)$. Therefore $w_{l-1}$ is a vertex at distance $l-1$ from $e$, and by inductive hypothesis $\alpha(w_{l-1})$ has been defined. Then $\alpha(w)$ is defined to be the (unique) vertex such that, if the edge $\{w_{l-1},w_l\}$ has label $x$, then the edge $\{\alpha(w_{l-1}),\alpha(w)\}$ has label $\sigma(x)$.

We claim that this is well-defined, that is the assignment does not depend on the path connecting $e$ to $w$. Indeed, consider another path $\gamma'=(e,w_1',w_2',\dots,w_l'=w)$ of length $l$ connecting $e$ to $w$ in $\C_\Gamma$, let $w_\gamma$ and $w_{\gamma'}$ be the two vertices obtained when considering $\gamma$ and $\gamma'$. We want to prove that $w_\gamma=w_{\gamma'}$. 

The claim is clear if $w_{l-1}=w_{l-1}'$, that is if the two paths share the last edge. It can also be shown that if the word that can be read on $\gamma'\circ \gamma^{-1}$ is of the form $(xy)^l$ for some $\{x,y\}\in E\Gamma$, then $w_\gamma=w_{\gamma'}$.
For the general case, suppose that the last edge of $\gamma$ and $\gamma'$ are different. From Theorem~\ref{Tits_theorem} we know that the words read on $\gamma$ and $\gamma'$ can be obtained from one another by applying finitely many braid relations. In particular, the label $x$ of $\{w_{l-1},w\}$ and the label $y$ of $\{w_{l-1}',w\}$ must be joined by an edge in $\Gamma$. Therefore, we can reduce this case to the previous, that is a tail subpath of $\gamma$ and of $\gamma'$ belong to a induced circle of length $2m_{x,y}$ in $\C_\Gamma$, and thus $w_\gamma=w_{\gamma'}$.

That is, we defined a bijection $\alpha$ of the vertices of the Cayley graph which preserves adjacency between vertices, that is an automorphism of the Cayley graph $\C_\Gamma$.

It is now easy to check that $\sigma\mapsto \alpha$ is a homomorphism of groups, and that it is indeed injective and surjective.
\end{proof}
With a slight abuse of notation, we will identify $\Aut(\Gamma)$ with the subgroup of $\A_\Gamma$ consisting of almost translations that fix the identity vertex.

Of course, there is nothing special about the identity vertex, and therefore we have:
\begin{lemma}\label{lemma21}
Let $\Gamma$ be a graph, $\sigma\in \Aut(\Gamma)$, and $v$ be a vertex in $\C_\Gamma$. Then there exists an automorphism $\alpha \in \A_\Gamma$ such that $\alpha(v)=v$ and such that $\sigma(\alpha,w)=\sigma$ for all vertices $w$ in $\C_\Gamma$.
\end{lemma}

Any almost translation that fixes a vertex in $\C_\Gamma$ is conjugated to an element of $\Aut(\Gamma)$, that is to an almost translation that fixes the identity vertex. Indeed, if $\alpha$ is such that $\alpha(v)=v$ for some $v\in \C_\Gamma$ then
\begin{equation*}
\alpha=L_{v}\beta L_{v^{-1}},    
\end{equation*} 
where $R_v\in W_\Gamma$ and $\beta\in\Aut(\Gamma)\leqslant \A_\Gamma$. Moreover, if $w$ is another fixed point of $\alpha$ then $ L_v \beta L_{v^{-1}} = L_w \beta L_{w^{-1}}$.
We have that $\sigma(\alpha,v)= \sigma(\beta, e)$.

In the next corollary we maintain the notation of Lemma \ref{lemma21}.
\begin{corollary}\label{corollary21}
Let $\Gamma$ be a graph, $\sigma\in \Aut(\Gamma)$, and $v$ be a vertex in $\C_\Gamma$. The automorphism $\alpha$ fixes pointwise the set
\begin{equation*}
\Bigl\{w\in V\Gamma\mid
\begin{array}{c}
\hbox{there is a path in $\C_\Gamma$ connecting $w$ to $v$}\\
\hbox{whose edges are all labelled by fixed points of $\sigma$}
\end{array}
\Bigr\}.  
\end{equation*}
\end{corollary}

\begin{lemma}\label{almtranslations_fg}
The subgroup of $\A_\Gamma$ generated by almost translations is isomorphic to $W_\Gamma\rtimes \Aut(\Gamma)$. In particular, if the simplicial graph $\Gamma$ is finite, then $W_\Gamma\rtimes \Aut(\Gamma)$ is finitely generated.
\end{lemma}
\begin{proof}
It is clear that $W_\Gamma$ and $\Aut(\Gamma)$, that is the subgroup of translations and the subgroup of almost translation fixing the identity vertex, intersect trivially.

Let us now consider an almost translation $\alpha$ that does not fix any vertex in $\C_\Gamma$, suppose it is not a translation, that is $\sigma(\alpha,-)\neq \id_\Gamma$, so that $\alpha(e)\neq e$. Let us consider the translation $L_{\alpha(e)}\in W_\Gamma$ that maps the vertex $e$ to the vertex $\alpha(e)$. Then we have that
\begin{equation*}
    \alpha= \bigl(L_{\alpha(e)}\cdot (L_{\alpha(e)})^{-1}\bigr)\cdot\alpha = L_{\alpha(e)}\cdot \bigl( (L_{\alpha(e)})^{-1}\cdot\alpha \bigr).
\end{equation*}
Notice that $(L_{\alpha(e)})^{-1}\cdot\alpha$ is an almost translation, and by construction it fixes the identity vertex, that is, it is an element of $\Aut(\Gamma)$. Thus, any almost translation in $\A_\Gamma$ can be expressed as a product of an element from $W_\Gamma$ and an element from $\Aut(\Gamma)$.

Moreover, the subgroup $W_\Gamma$ of translation is normal in the subgroup of $\A_\Gamma$ of almost translations. Indeed, if $\alpha$ is an almost translation and $L_w\in W_\Gamma$, then $\alpha^{-1}L_w\alpha = L_{\alpha(w)}$.

Therefore, the subgroup of almost translations splits as the semidirect product $W_\Gamma\rtimes \Aut(\Gamma)$.
If the graph $\Gamma$ is finite, then $W_\Gamma$ is finitely generated and $\Aut(\Gamma) $ is finite, so that $W_\Gamma\rtimes \Aut(\Gamma)$ must be finitely generated as well. So, the lemma is proved.
\end{proof}

We finish this section with three examples, the first of a two-ended right-angled Coxeter group, the second of a one-ended right-angled Coxeter group, and the third of a infinitely-ended one. The first two examples have Cayley graphs with countable automorphism group, whereas the third has a Cayley graph with uncountable automorphism group.

\begin{example}
As an example of two-ended right-angled Coxeter group, let us consider the graph $\Gamma$
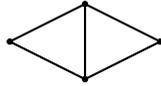
\begin{equation*}
\begin{tikzpicture}
\draw[thick] (-.5,0) --(.5,-.5) -- (1.5,0) -- (.5,.5) -- (-.5,0)

(.5,.5) -- (.5,-.5);

\draw[fill] 
(-.5,0) circle [radius=0.035]
(.5,-.5) circle [radius=0.035]
(1.5,0) circle [radius=0.035]
(.5,.5) circle [radius=0.035];
\end{tikzpicture}
\end{equation*}
that is a complete graph on four vertices with an edge removed, all of whose edges have weight two.
Its Cayley graph is
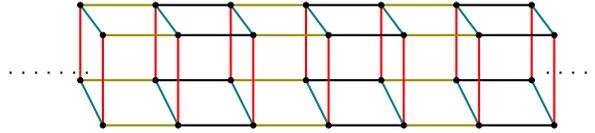
\begin{equation*}
\begin{tikzpicture}
    \foreach \i in {-3,...,3} {
\draw[thick,red] (\i,.5)--(\i,-.5);}
\foreach \i in {-3,-1,1} {
\draw[thick,olive] (\i,.5)--(\i+1,.5);
\draw[thick,olive] (\i,-.5)--(\i+1,-.5);}
\foreach \i in {-2,0,2} {
\draw[thick] (\i,.5)--(\i+1,.5);
\draw[thick] (\i,-.5)--(\i+1,-.5);}
\draw[thick, loosely dotted] (-2.9,-.4) --(-4,-.4);
\draw[thick, loosely dotted] (3.2,-0.4) --(4,-.4);

\foreach \i in {-3,...,3} {
\draw[thick,red] (.3+\i,.1)--(.3+\i,-1.1);}
\foreach \i in {-3,-1,1} {
\draw[thick,olive] (.3+\i,.1)--(1.3+\i,.1);
\draw[thick,olive] (.3+\i,-1.1)--(\i+1.3,-1.1);}
\foreach \i in {-2,0,2} {
\draw[thick] (.3+\i,.1)--(\i+1.3,.1);
\draw[thick] (.3+\i,-1.1)--(\i+1.3,-1.1);}

\foreach \i in {-3,...,3} {
\draw[thick,teal] 
(\i,.5)--(\i+.3,.1)
(\i,-.5)--(\i+.3,-1.1);
}

\filldraw (0,0.5) circle[radius=0.035] 
(1,0.5) circle[radius=0.035]
(2,0.5) circle[radius=0.035]
(3,0.5) circle[radius=0.035]
(-1,0.5) circle[radius=0.035]
(-2,0.5) circle[radius=0.035]
(-3,0.5) circle[radius=0.035]

(0,-0.5) circle[radius=0.035] 
(1,-0.5) circle[radius=0.035] 
(2,-0.5) circle[radius=0.035]
(3,-0.5) circle[radius=0.035]
(-1,-0.5) circle[radius=0.035] 
(-2,-0.5) circle[radius=0.035]
(-3,-0.5) circle[radius=0.035]

(0.3,.1) circle[radius=0.035] 
(1.3,.1) circle[radius=0.035] 
(2.3,.1) circle[radius=0.035]
(3.3,.1) circle[radius=0.035]
(-.7,.1) circle[radius=0.035] 
(-1.7,.1) circle[radius=0.035]
(-2.7,.1) circle[radius=0.035]

(0.3,-1.1) circle[radius=0.035]
(1.3,-1.1) circle[radius=0.035]
(2.3,-1.1) circle[radius=0.035]
(3.3,-1.1) circle[radius=0.035]
(-.7,-1.1) circle[radius=0.035] 
(-1.7,-1.1) circle[radius=0.035]
(-2.7,-1.1) circle[radius=0.035];
\end{tikzpicture}
\end{equation*}
and the automorphism $\A_\Gamma$ is countable by Proposition \ref{prop_countable}.

\end{example}

\begin{example}\label{example_goodsepset}
Suppose that $\Gamma=C_4$ is a cycle of lenght four all of whose edges have weight two (that is, the corresponding Coxeter group is right-angled), and $\sigma\in\Aut(\Gamma)$ is the reflection along the dashed axis
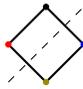
\begin{equation*}
\begin{tikzpicture}
\draw[thick] 
(0,0) --(.5,.5) 
(1,0) --(.5,-.5);
\draw[thick] 
(.5,.5) --(1,0) 
(0,0) -- (.5,-.5);
\draw[dashed] (0,-.5)--(1,.5);
\draw[fill,red] 
(0,0) circle [radius=0.035];
\draw[fill,blue] 
(1,0) circle [radius=0.035];
\draw[fill] 
(.5,.5) circle [radius=0.035];
\draw[fill,olive] 
(.5,-.5) circle [radius=0.035];
\end{tikzpicture}
\end{equation*}
Let $v$ be a vertex in $\C_\Gamma$ (in the picture $v$ is at distance one from the identity vertex), which in this case is the infinite two-dimensional grid depicted in the following picture. The automorphism $\sigma$ induces the automorphism $ \alpha\in \C_\Gamma$, which is the reflection along the dashed diagonal passing through $v$, such that $\sigma(\alpha,v)=\sigma$.
\begin{figure}[ht!]
\begin{tikzpicture}
    \foreach \i in {-2,...,2} {
\draw[thick,blue] (\i,0)--(\i,-1);

\draw[thick,blue] (\i,2)--(\i,1);

\draw[thick,red] (\i,1)--(\i,0);

\draw[thick,red] (\i,3)--(\i,2);
}

\foreach \i in {-1,...,3} {
\draw[thick] (0,\i) --(1,\i);

\draw[thick] (-2,\i) --(-1,\i);

\draw[thick,olive] (1,\i) --(2,\i);

\draw[thick,olive] (-1,\i) --(0,\i);
}

\foreach \i in {-2,...,2}
\foreach \j in {-1,...,3} {
\filldraw (\i,\j) circle[radius=0.035];}

\draw[dashed] (-2.5,-1.5)--(2.5,3.5);
\draw[dashed] (2,-2)--(-2.5,2.5);

\draw[thick, loosely dotted] (-2.2,1) --(-3,1);
\draw[thick, loosely dotted] (2.2,1) --(3,1);
\draw[thick, loosely dotted] (0,-1.2) --(0,-2);
\draw[thick, loosely dotted] (0,3.2) --(0,4);

\filldraw (0,0) node[anchor=south east] {\small $e$};
\filldraw (0,1)  node[anchor=south east] {\small $v$};
\filldraw (0,2);
\filldraw (0,-1);

\filldraw (1,0);
\filldraw (1,1);
\filldraw (1,2);
\filldraw (1,-1);

\filldraw (-1,0);
\filldraw (-1,1);
\filldraw (-1,2);
\filldraw (-1,-1);

\filldraw (-2,0);
\filldraw (2,0);
\end{tikzpicture}
\end{figure}
If $\beta$ is the reflection along the dashed diagonal passing through $e$, then $\alpha=L_v \beta L_{v^{-1}}$. The axes of the two reflexions $\alpha$ and $\beta$ are perpendicular because the automorphism $L_v$ fixes the edge $\{e,v\}$ inverting the two vertices, that is $L_v(e)=v$ and $L_v(v)=e$. Notice that $\sigma(\alpha,v)= \sigma(\beta,e)$.

More generally, if we consider a cycle $C_n$ of length $n\geqslant 4$, we obtain a tessellation of the plane by squares.

In the Cayley graph $\Cay(C_5)$ at any vertex $v$ there are five, that is $\lvert V{C_5}\rvert$, incident vertices that form commuting squares (corresponding to edges of the cycle $C_5$):
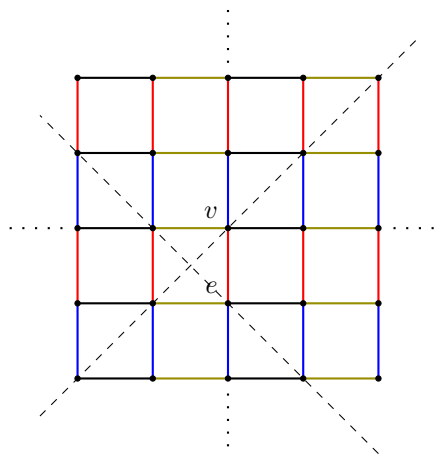
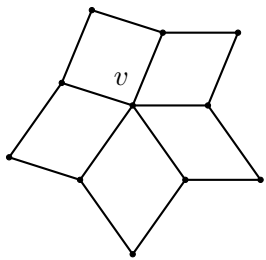
\begin{equation*}
\begin{tikzpicture}
\draw[thick] (-1,0) --(0,0) --(.4,{ sqrt(15)/4}) -- (-.6,{ sqrt(15)/4}) --(-1,0)

(0,0) -- (.7,{ -{sqrt(48)}/7}) --(-.3,{ -{sqrt(48)}/7}) --(-1,0)

(-1,0) --(-1.7,{ -{sqrt(48)}/7}) --(-1,{ -{2* sqrt(48)}/7}) --(-.3,{ -{sqrt(48)}/7})

(-1,0) --({-sqrt(8)/3 -1},.3) -- ({-sqrt(8)/3 -.6},{ sqrt(15)/4 +.3}) -- (-.6,{ sqrt(15)/4})

(-1.7,{ -{sqrt(48)}/7}) --({-sqrt(8)/3 -1.7}, { -{sqrt(48)}/7 +.3}) --({-sqrt(8)/3 -1},.3)
;

\draw[fill] 
(-1,0) circle [radius=0.035] node[yshift=.35cm ,xshift=-.15cm] {$v$}
(0,0) circle [radius=0.035]
(.4,{ sqrt(15)/4}) circle [radius=0.035]
(-.6,{ sqrt(15)/4}) circle [radius=0.035]
(.7,{ -{sqrt(48)}/7}) circle [radius=0.035]
(-.3,{ -{sqrt(48)}/7}) circle [radius=0.035]
(-1.7,{ -{sqrt(48)}/7}) circle [radius=0.035]
(-1,{ -{2* sqrt(48)}/7}) circle [radius=0.035]
({-sqrt(8)/3 -1},.3) circle [radius=0.035]
({-sqrt(8)/3 -.6},{ sqrt(15)/4 +.3}) circle [radius=0.035]
({-sqrt(8)/3 -1.7}, { -{sqrt(48)}/7 +.3}) circle [radius=0.035];
\end{tikzpicture}
\end{equation*}
This pattern covers a plane and the automorphism group of such Cayley graphs is finitely generated, and therefore countable, as we will deduce from Proposition \ref{prop_countable}.
\end{example}

\begin{example}\label{example_part1}
Let us now consider the disconnected graph~$\Gamma$
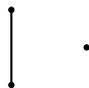
\begin{equation*}
\begin{tikzpicture}
\draw[thick] 
(0,-.5) --(0,.5);
\draw[fill] 
(0,-.5) circle [radius=0.035];
\draw[fill] 
(0,.5) circle [radius=0.035];
\draw[fill] 
(1,0) circle [radius=0.035];
\end{tikzpicture}
\end{equation*}
where the only edge has weight two.
Let $\Gamma_1$ denote the connected component with two vertices and one edge, and let $\Gamma_2$ be the connected component consisting of one vertex. We have that $\Aut(\Gamma_1)=\{e,\sigma\}$ is the cyclic group of order two and that $\C_{\Gamma_1})$ is a commuting square.
Let $v$ be a vertex of $\C_{\Gamma_1})$. In view of Corollary~\ref{corollary21}, the automorphism $\sigma$ (which switches the two vertices of $\Gamma_1$) induces an automorphism $\alpha\in \Aut\bigl(\C_{\Gamma_1})\bigr)$ fixing the vertex $v$, the reflection along the following dashed line (and along a diagonal in each square in the Cayley graph):
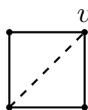
\begin{equation*}
\begin{tikzpicture}
\draw[thick] 
(.5,-.5) --(-.5,-.5) 
(-.5,.5) -- (.5,.5)
(.5,.5) --(.5,-.5) 
(-.5,-.5) --(-.5,.5);
\draw[dashed,thick] (-.5,-.5)--(.5,.5);
\draw[fill] (.5,.5) circle [radius=0.035] node[anchor=south] {$v$}
(-.5,.5) circle [radius=0.035] 
(-.5,-.5) circle [radius=0.035] 
(.5,-.5) circle [radius=0.035];
\end{tikzpicture}
\end{equation*}
Colouring each generator of $W_\Gamma$ with a different colour (blue for the vertex in $\Gamma_2$ red and olive for the two vertices in $\Gamma_1$), we have the following situation
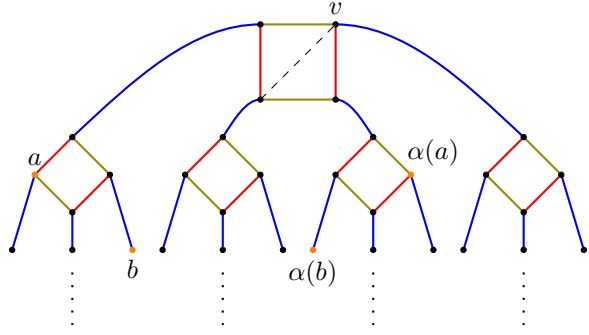
\begin{equation*}
\begin{tikzpicture}
\draw[dashed] 
(-.5,-.5)--(.5,.5);

\draw[thick,red] 
(.5,.5) --(.5,-.5) 
(-.5,-.5) --(-.5,.5)
(-1,-1) --(-1.5,-1.5)
(-1,-2) --(-.5,-1.5)
(1,-2) --(1.5,-1.5)
(1,-1) --(.5,-1.5)
(-3,-1) --(-3.5,-1.5)
(-3,-2) --(-2.5,-1.5)
(3,-1) --(2.5,-1.5)
(3,-2) --(3.5,-1.5);

\draw[thick,olive]
(-.5,.5) -- (.5,.5)
(.5,-.5) --(-.5,-.5)
(-1.5,-1.5) --(-1,-2)
(-.5,-1.5) --(-1,-1)
(.5,-1.5) --(1,-2)
(1.5,-1.5) --(1,-1)
(-3.5,-1.5) --(-3,-2)
(-2.5,-1.5) --(-3,-1)
(2.5,-1.5) --(3,-2)
(3.5,-1.5) --(3,-1)
;

\draw[thick,blue] 
(-.5,-.5) cos(-1,-1) 
(.5,-.5) cos(1,-1) 
(-.5,.5) cos(-3,-1) 
(.5,.5) cos(3,-1);

\draw[blue,thick]
(-3.5,-1.5) --(-3.8,-2.5)
(-3,-2) --(-3,-2.5)
(-2.5,-1.5) --(-2.2,-2.5)

(-1.5,-1.5) --(-1.8,-2.5)
(-1,-2) --(-1,-2.5)
(-.5,-1.5) --(-.2,-2.5)

(.5,-1.5) --(.2,-2.5)
(1,-2) --(1,-2.5)
(1.5,-1.5) --(1.8,-2.5)

(2.5,-1.5) --(2.2,-2.5)
(3,-2) --(3,-2.5)
(3.5,-1.5) --(3.8,-2.5);

\draw[loosely dotted, thick] (-3,-2.8)--(-3,-3.5)
(-1,-2.8)--(-1,-3.5)
(1,-2.8)--(1,-3.5)
(3,-2.8)--(3,-3.5);

\draw[fill] (.5,.5) circle [radius=0.035] node[anchor=south] {\small $v$}
(-.5,.5) circle [radius=0.035] 
(-.5,-.5) circle [radius=0.035] 
(.5,-.5) circle [radius=0.035] 
(-3,-1) circle [radius=0.035] 
(-1,-1) circle [radius=0.035] 
(1,-1) circle [radius=0.035] 
(3,-1) circle [radius=0.035] 
(-2.5,-1.5) circle [radius=0.035] 
(-1.5,-1.5) circle [radius=0.035] 
(-.5,-1.5) circle [radius=0.035] 
(.5,-1.5) circle [radius=0.035]
(2.5,-1.5) circle [radius=0.035]
(3.5,-1.5) circle [radius=0.035]
(-3,-2) circle [radius=0.035]
(-1,-2) circle [radius=0.035]
(1,-2) circle [radius=0.035]
(3,-2) circle [radius=0.035]
(-3,-2.5) circle [radius=0.035]
(-1,-2.5) circle [radius=0.035]
(1,-2.5) circle [radius=0.035]
(3,-2.5) circle [radius=0.035]
(-3.8,-2.5) circle [radius=0.035]
(3.8,-2.5) circle [radius=0.035]
(2.2,-2.5) circle [radius=0.035]
(-1.8,-2.5) circle [radius=0.035]
(-.2,-2.5) circle [radius=0.035]
(1.8,-2.5) circle [radius=0.035];

\draw[fill,orange] 
(-3.5,-1.5) circle [radius=0.035] 
node[anchor=south] {\small ${\textcolor{black}a}$}
(1.5,-1.5) circle [radius=0.035]
node[xshift=.3cm,yshift=.3cm] {\small ${\textcolor{black}{\alpha(a)}}$}
(-2.2,-2.5) circle [radius=0.035]
node[anchor=north] {\small ${\textcolor{black}b}$}
(.2,-2.5) circle [radius=0.035]
node[anchor=north] {\small ${\textcolor{black}{\alpha(b)}}$};
\end{tikzpicture}
\end{equation*}
where we depicted the vertices $a$ and $b$ and their respective images under the automorphism $\alpha$.
\end{example}

\section{Good separating sets}\label{section3}
In this section we prove one implication of Theorem~\ref{thmA}. A fundamental notion allowing us to do so is the one of a \emph{good separating set}.

\begin{definition}\label{goodseparatingset}
Let  $S\subsetneq \Gamma$ be a proper separating set, so that $\Gamma\setminus S=C_1\sqcup \dots \sqcup C_n$ is the disjoint union of $n\geqslant 2$ connected components. Suppose that there exists a non-empty set $ I\subsetneq \{1,\dots,n\}$ with the property that, fixing $\Gamma_1:= S\sqcup\bigsqcup_{i\in I} C_i$ and $\Gamma_2:= S \sqcup \bigsqcup_{i\notin I}C_i$, such that there exists a non-trivial $\alpha\in \Aut(\Gamma_1)$ such that $\alpha\restriction_S=\id_S$.
Then we call $S$ a \emph{good separating set}.
\end{definition}
Notice that $S$ can be the empty set. For instance, the only good separating set for the graph considered in Example~\ref{example_part1} is the empty set. The idea behind the definition is the following. In the presence of a proper separating set~$S$, the Coxeter group $W_\Gamma$ splits properly as the amalgamated product $W_\Gamma= W_{\Gamma_1}\ast_{W_S}W_{\Gamma_2}$, where $\Gamma_1$ and $\Gamma_2$ are given as in Definition \ref{goodseparatingset}. If the separating set $S $ is a good separating set, then we can permute non-trivially the canonical generators of $\Gamma_1$, fixing pointwise the ones of $W_S$. The geometric intuition is that we are able to produce non-trivial automorphisms of any coset of $W_{\Gamma_1}$ inside (the Cayley graph of) $W_\Gamma$ that extend to the rest of the graph $\C_\Gamma$. In Corollary~\ref{cor:uncountable stabilisers} we will exploit this idea to prove that in the presence of a good separating set the automorphism of the Cayley graph is uncountable.

We can characterise good separating sets in terms of stars of vertices in the graph $\Gamma$. In the following lemma we maintain the notation of Definition~\ref{goodseparatingset}.
\begin{lemma}\label{char_good_sep_set}
Let $\Gamma$ be a finite simplicial graph. The following conditions are equivalent:
\begin{enumerate}
\item there exists a vertex $v\in V\Gamma$ and a non-trivial $\alpha\in\Aut(\Gamma)$ such that $\alpha\restriction_{\star(v)}=\id_\Gamma$;
\item $\Gamma$ admits a good separating set;
\item there exist distinct elements $\alpha,\beta \in \Aut(\Gamma_1)$ such that $\alpha\restriction_S=\beta\restriction_S$.
\end{enumerate}
\begin{proof}
Suppose that there exists a vertex $v\in V\Gamma$ such that $\alpha\restriction_{\star(v)}=\id$. As $\alpha\neq \id_\Gamma$, we have that $\Gamma\setminus \star(v)\neq \emptyset$. Therefore, $\link(v)$ is a good separating set, and the first condition implies the second.

On the other hand, suppose that there exists a good separating set $S$ with non-trivial associated automorphism $\alpha\in\Aut(\Gamma_1)$ such that $\alpha\restriction_S=\id$. For any vertex $v\in \Gamma_2\setminus S$ we have that $\alpha\restriction_{\star(v)}=\id$, that is, the second condition implies the first.

To conclude the proof, the second condition implies the third, where $\alpha$ is the non-trivial automorphism provided by the good separating set $\beta=\id_\Gamma$, and the third condition implies the second considering the non-trivial automorphism $\alpha^{-1}\beta\in\Aut(\Gamma)$.
\end{proof}
\end{lemma}

The following will be the key lemma to prove the main result of this section.
\begin{lemma}\label{keylemma}
Let $\alpha\in\A_\Gamma$, let $v$ be a vertex in $\C_\Gamma$ and $x\in V\Gamma$. Then
\begin{equation*}
    \sigma(\alpha,v)\restriction_{\star(x)}=\sigma(\alpha,vx)\restriction_{\star(x)}.
\end{equation*}
\begin{proof}
Let $y\in \link(x)$, so that in the Cayley graph $\C_\Gamma$ we see an relation cycle of length $2m(x,y)$ incident at $v$ and labelled alternatively by $x$ and~$y$ (in Figure \ref{fig_relationcycle_commutingsquare} the general case on the left-hand side, and the right-angled case, that is $m(x,y)=2$, on the right-hand side).
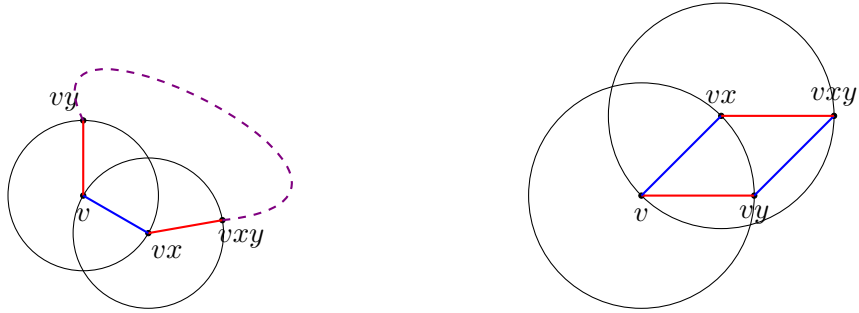
\begin{figure}[ht!]
\begin{tikzpicture}
\draw [ name path    = great circle 1]
        (0,0) circle (1cm);

\path [ name path    = great circle 2]
        (30:1) circle (1cm);
\path [name intersections={of=great circle 1 and great circle 2}] ;
\coordinate (a) at (intersection-1);
\coordinate (b) at (intersection-2);

\draw [ name path    =circleb](b) circle (1cm);

\path [ name path    = great circle 3]
       (70:1)+(b) circle (1cm);
\path [name intersections={of= circleb and great circle 3}];
\path (intersection-1) circle (1cm);

\filldraw (0,0) circle (0.035) node[yshift=-.25cm] {$v$}
(a) circle (0.035) node[yshift=.25 cm,xshift=-.25cm] {$vy$}
(b) circle (0.035) node[yshift=-.25 cm,xshift=.25cm] {$vx$}
(intersection-1) circle (0.035) node[yshift=-.25 cm,xshift=.25cm] {$vxy$}
;

\draw[thick,red] (0,0)-- (a)
(b)-- (intersection-1);
\draw[thick,blue] (0,0)-- (b);
\draw[thick, violet, dashed]
(intersection-1) 
.. controls (5,0) and (-1,3) .. 
(a);
\end{tikzpicture}
\qquad
\begin{tikzpicture}
\draw (0,0) circle (1.5cm);
\draw ({1.5*cos(45)},{(1.5*cos(45))}) circle (1.5cm);

\filldraw (0,0) circle (0.035) node[yshift=-.25cm] {$v$}
(1.5,0) circle (0.035) node[yshift=-.25 cm] {$vy$};

\begin{scope}[rotate=45]
\filldraw (1.5,0) circle (0.035) node[yshift=.25 cm] {$vx$};

\begin{scope}[shift={(1.5,0)},rotate=45]
\filldraw (0,-1.5) circle (0.035) node[yshift=.25 cm] {$vxy$};
\end{scope}
\end{scope}

\draw[thick,red] (0,0) -- (1.5,0);

\begin{scope}[rotate=45]
\draw[thick,blue] (0,0) -- (1.5,0);
\end{scope}

\begin{scope}[shift={(1.5,0)},rotate=45]
\draw[thick,blue] (0,0) -- (1.5,0);
\end{scope}

\begin{scope}[shift={({1.5*cos(45)},{(1.5*cos(45))})}]
\draw[thick,red] (0,0) -- (1.5,0);
\end{scope}
\end{tikzpicture}
\caption{Relation cycle, commuting square.}
\label{fig_relationcycle_commutingsquare}
\end{figure}

\noindent
The edges $\{v,vy\}$ and $\{vx,vxy\}$ are labelled by $y$. Moreover, as $\alpha$ is an automorphism of the graph~$\C_\Gamma$, it must map the relation cycle to a relation cycle of the same length, and in particular $\sigma(\alpha,v)(y)=\sigma(\alpha,vx)(y)$, that is $\sigma(\alpha,v)$ and $\sigma(\alpha,vx)$ coincide on $\link(x)$.

Clearly, we also have that $\sigma(\alpha,v)(x)=\sigma(\alpha,vx)(x)$, because the vertices~$v$ and~$vx$ are joined in $\C_\Gamma$ by the edge labelled by $x$. Therefore the lemma is proved.
\end{proof}
\end{lemma}

In view of Lemma~\ref{char_good_sep_set}, the presence of a good separating set in $\Gamma$ is equivalent to the condition appearing in the statement of Theorem \ref{thmA}, that is, to the existence of some $x\in V\Gamma$ and some non-trivial $\alpha\in \Aut(\Gamma)$ such that $\alpha\restriction_{\star(x)}=\id_\Gamma$.

Thus, in Proposition~\ref{prop_countable} we prove that if the graph $\Gamma$ is finite and does not admit any good separating set, then the automorphism group $\A_\Gamma$ only admits finite vertex stabilisers. Notice that, in general, the group $\A_\Gamma$ can be uncountable even if the graph $\Gamma$ is finite (Example \ref{example_part1} provides an example of an uncountable $\mathcal{A}_\Gamma$ with $\lvert V\Gamma \rvert =3$).

\begin{proposition}\label{prop_countable}
Let $\Gamma$ be a connected, simplicial graph, and suppose that there is no good separating set in $\Gamma$. Then all elements in $\A_\Gamma$ are almost translations, that is $\A_\Gamma=W_\Gamma\rtimes \Aut(\Gamma)$. In particular, if $\Gamma$ is finite then the automorphism group is finitely generated, and therefore countable.
\end{proposition}
\begin{proof}
Let $\alpha\in \A_\Gamma$, $v$ be a vertex in $\C_\Gamma$ and $x\in V\Gamma$. By Lemma \ref{keylemma} we have that $\bigl(\sigma(\alpha,v)^{-1}\circ \sigma(\alpha,vx)\bigr)\restriction_{\star(x)}=\id_\Gamma$, and it must then be that $\sigma(\alpha,v)= \sigma(\alpha,vx)$, because by Lemma~\ref{char_good_sep_set} for all non-trivial $\sigma\in \Aut(\Gamma)$ and for all $x\in V\Gamma$ we have that $\sigma\restriction_{\star(x)}\neq\id_\Gamma$.

Therefore, for all $x\in V\Gamma$ we have that $\sigma(\alpha,v)=\sigma(\alpha,vx)$. Repeating this process and applying it inductively on the vertices of the spheres of increasing radius, we conclude that $\alpha$ is an almost translation.

Whenever $\Gamma$ is a finite graph, $\A_\Gamma$ is finitely generated, and therefore countable, in view of Corollary \ref{almtranslations_fg}.
\end{proof}

\section{Configurations of local actions}\label{section4}
In this section we will prove the converse implication of Theorem \ref{thmA}, that is that the existence of a good separating set in $\Gamma$ forces the automorphism group $\mathcal{A}_\Gamma$ to be uncountable.

The following two notions will be helpful to do so.

\begin{definition}
We say that $\overline{\sigma} \in \Aut(\Gamma)^{W_\Gamma}$ is a \emph{legal configuration of local actions} (a \emph{legal configuration} for short) if  there exists $\alpha \in \A_\Gamma$ such that $\overline{\sigma}(v) = \sigma(\alpha, v)$ for all $v \in W_\Gamma$. We will use the symbol $\Sigma_\Gamma$ to denote the set of all legal configurations in $ \Aut(\Gamma)^{W_\Gamma}$.

We say that $\overline{\sigma} \in \Aut(\Gamma)^{W_\Gamma}$ \emph{satisfies the $(*)$-condition} if for all $v \in W_\Gamma$ and all $x \in V\Gamma$ we have
    \begin{displaymath}
        \overline\sigma(v)\restriction_{\star(x)} = \overline\sigma(vx)\restriction_{\star(x)}.
    \end{displaymath}
\end{definition}

In particular, given an automorphism $\alpha\in \mathcal{A}_\Gamma$, by definition, the element 
$\overline\sigma_\alpha:=\{\sigma(\alpha,u)\mid u\in W_\Gamma\}\in \Aut(\Gamma)^{W_\Gamma}$ is a legal configuration of local actions.
On the other hand, Lemma \ref{keylemma} assures that it satisfies the $(*)$-condition.

We will see in the following result that these two notions, that is the one of legal configurations and the $(*)$-condition, are equivalent.

\begin{lemma}
    \label{lemma:legal configurations}
    Let $\overline{\sigma} \in \Aut(\Gamma)^{W_\Gamma}$ be arbitrary. Then $\overline{\sigma} \in \Sigma_\Gamma$ if and only if $\overline{\sigma}$ satisfies the $(*)$-condition.
\end{lemma}
\begin{proof}
If $\overline{\sigma} \in \Sigma_\Gamma$ then, by definition, there exists $\alpha\in \mathcal{A}_\Gamma$ such that $\overline\sigma(v)=\sigma_\Gamma(\alpha,v)$ for all $v\in W_\Gamma$. Lemma \ref{keylemma} assures us that $\overline\sigma$ satisfies the $(*)$-condition.    
 
    Now, suppose that $\overline\sigma \in \Aut(\Gamma)^{W_\Gamma}$ satisfies the $(*)$-condition. We will construct $\alpha \in \A_\Gamma$ such that $\sigma_\Gamma(\alpha, v) = \overline{\sigma}(v)$ for all $v\in W_\Gamma$. First, we inductively construct a sequence of maps $\alpha_n \colon B(e, n) \to B(e, n)$, where $B(e, n)$ is the ball of radius $n$ centred at the identity vertex $e$ in the graph~$\C_\Gamma$.
    
    For $n = 0$ we simply have $\alpha_0(e) := e$. For $n \geqslant 1$ we inductively define 
    $\alpha_{n}(v) := \alpha_{n-1}(v)$ if $v \in B(e, n-1)$. On the other hand, if $v \in \partial B(e, n)$, as the Cayley graph is connected, there exists (at least) a generator $x\in V\Gamma$ such that $vx\in B(e,n-1)$; in this case we define
    \begin{displaymath}\alpha_{n}(v):=\underbrace{\alpha_{n-1}(vx)}_{\in V\mathcal{C}_\Gamma}\underbrace{\overline{\sigma}(vx)(x)}_{\in V\Gamma}\in V\mathcal{C}_\Gamma\end{displaymath}
%
    By definition, the vertex $\alpha_{n}(v)$ is the unique vertex in the Cayley graph $\mathcal{C}_\Gamma$ connected to the vertex $\alpha_{n-1}(vx)$ by an edge labelled by $\overline{\sigma}(vx)(x)$.
    
    We will use induction on $n$ to show that $\alpha_n\colon B(e,n)\to B(e,n)$ is a graph automorphism for all natural numbers $n\in \mathbb{N}$, that is, it is a bijection of the set $B(e,n)$ that preserves edges.

    Clearly, for $n=0$ the claim holds. Now, let us assume that $\alpha_i$ is a well-defined graph automorphism for all $i \leq n$, and consider $\alpha_{n+1} \colon B( e , n+1) \to B( e , n+1)$.
    
    First, we show that $\alpha_{n+1}$ is well-defined on all vertices of $B( e , n+1 )$. By the induction hypothesis $\alpha_{n+1}$ is well-defined on $B(e , n)$; thus, let $w \in {\partial B( e , n+1 )}$ and suppose that there are two distinct generators $x,y \in V\Gamma$ and $u,v \in  B( e , n )$ such that $ux = w = vy$. First, we show that
    \begin{displaymath}
        \alpha_n(v)\overline{\sigma}(v)(x) = \alpha_n(u)\overline{\sigma}(u)(y).
    \end{displaymath}
    To this end, let $W_u=(x_1, \dots, x_n)$ and $W_v = (y_1, \dots, y_n)$ be two reduced words representing $u$ and $v$, respectively (thus, ${x_1, y_1,} \dots, {x_n, y_n \in V\Gamma}$). It immediately follows that both $W = (x_1, \dots, x_n, x)$ and $W' = (y_1, \dots, y_n, y)$ are reduced words representing $w$. By Theorem \ref{Tits_theorem}, we see that $W'$ can be obtained from $W$ by applying a finite number of braid relations. In particular, we see that $\{x,y\} \in E\Gamma$: let $m = m(x,y)$ be the corresponding weight. 
    Thus, we have a relation cycle of $x$ and $y$ at the vertex $w\in \mathcal{C}_\Gamma$, given by the following vertices:

    \begin{displaymath}
    \begin{split}
        u_0 = u,\qquad u_1 = u y,&\qquad u_2 = u yx,\qquad\dots\\
        u_{2m-3} = u (yx)^{m-2}y,\qquad  &u_{2m-2} = u(yx)^{m-1}.
    \end{split}
    \end{displaymath}
    We have that 
    \begin{displaymath}
        u_{2m-2} = u (yx)^{m-1}= u (yx)^{m}xy = uxy = wy = v
    \end{displaymath}
    and that $u_0, u_1, \dots, u_{2m-2} \in B(e,n)$.
    Clearly, the vertices $u_0, u_1, \dots, u_{2m-2}$ form a path in $B( e , n)$ labelled by $y$ and $x$ (it is the relation cycle of $x$ and $y$ at $w$). In particular, we see that
    \begin{align*}
        \overline{\sigma}(u_0)(x) = \dots = \overline{\sigma}(u_{2m_1})(x) = x'\\
        \overline{\sigma}(u_0)(y) = \dots = \overline{\sigma}(u_{2m_1})(y) = y'
    \end{align*}
    because $\overline{\sigma}$ satisfies the $(*)$-condition.
    Clearly, $\{x',y'\} \in E\Gamma$ as well, with $m(x',y') = m$. Consequently, we see that the path 
    \begin{displaymath}
        P = (u_0, u_1, \dots, u_{2m-2}) \subseteq B( e,n+1)
    \end{displaymath}
    with alternating labels $y,x$ maps onto a path 
    \begin{displaymath}
        P' = \bigl(\alpha_n(u_0), \alpha_n(u_1), \dots, \alpha_n(u_{2m-2})\bigr) \subseteq B( e,n+1)
    \end{displaymath}
    with alternating labels $y',x'$ and that the path $P'$ uniquely completes to a circuit with alternating labels $x',y'$. In particular
    \begin{displaymath}
        \alpha_{n+1}(ux) = \alpha_n(u)\overline{\sigma}(u)(x) = \alpha_n(v)\overline{\sigma}(y) = \alpha_{n+1}(vy).
    \end{displaymath}
    We see that $\alpha_{n+1} \colon B( \id,n+1) \to B( \id,n+1)$ is well-defined.

    To see that $\alpha_{n+1}$ preserves edges, we need to show that for all $u,v \in B(e,n+1)$ we have that $\{\alpha_{n+1}(u),\alpha_{n+1}(v)\} \in E\C_\Gamma$ whenever $\{u,v\}\in E\C_\Gamma$. From the definition of $\alpha_{n+1}$ we clearly see that if $v = ux$ and $\|u\| < \|v\|$, then the vertices $\alpha_{n+1}(u)$ and $\alpha_{n+1}(v)$ are connected by an edge labelled by $\overline{\sigma}(u)(x)$. 
    
    Now assume that there are $u,v \in \partial B(e,n+1)$ such that $\|u\| = \|v\|$ and that there is $x \in V\Gamma$ such that $u = vx$. Let $W_u=(x_1, \dots, x_n)$ and $W_v = (y_1, \dots, y_n)$, where $x_1, y_1, \dots, x_n, y_n \in V\Gamma$, be some reduced expressions for $u$ and $v$, respectively. We see that the word $W = (y_1, \dots, y_n, x, x_n^{-1}, \dots, x_1^{-1})$ is an expression for $vxu^{-1} = 1$. Clearly, $\lvert W\rvert  = 2n + 1$. However, the elementary simplification on $(s_1)$ shortens the word by exactly $2$ letters, meaning that the word $W$ can never be reduced to the empty string, which is a contradiction. We see that no such $u,v$ and $x$ can occur and, therefore, the map $\alpha_{n+1}$ preserves edges.
    
    Finally, we show that $\alpha_{n+1}$ is a bijection. Let us note that it is clear from the construction that $\alpha_{n+1}$ is injective on $B(w,1)$ for every $w \in B( e,n)$. Now suppose that there are $u,v \in B(e,n+1)$ such that $\alpha_{n+1}(u) = \alpha_{n+1}(v)$. Since $\alpha_{n+1}$ is injective on $B(e,n)$, we see that $\{u,v\} \cap \partial B(e, n+1) \neq \emptyset$.    Without loss of generality we may assume that $u \in B(e,n+1)$; let $u' \in B(e,n)$ and $x \in V\Gamma$ be such that $u = u'x$. Clearly $\alpha_{n+1}(u) \in B\bigl(\alpha_{n}(u'),1\bigr)$, so $n-1 \leq \|\alpha_{n+1}(u)\| \leq n+1$. 
    
    Suppose that $\|\alpha_{n+1}(u)\| = n$, and let $(x_1', \dots, x_n')$ and $(y_1', \dots, y_n')$, where $x_1', y_1', \dots, x_n', y_n' \in V\Gamma$ be some reduced expressions for $\alpha_{n+1}(u)$ and $\alpha_{n+1}(u')$. It then follows that the expression
    \begin{displaymath}
        W = (x_1', \dots, x_n', \overline{\sigma}(u')(x),y_n', \dots, y_1')
    \end{displaymath}
    represents the trivial element and therefore can be reduced to the empty string, which is a contradiction because the simplification $(s_1)$ always reduces by 2 syllables and $|W| = 2n+1$. Thus $\|\alpha_{n+1}(u)\| \neq n$. 
    
    To see that $\|\alpha_{n+1}(u)\| = n+1$ we consider the partition $V\Gamma = V_+ \cup V_-$, where $z \in V_+$ if $\|u'z\| > \|u'\|$ and $z \in V_-$ if $\|u'z\| < \|u'\|$. As $\alpha_{n+1}$ must stabilise set-wise the ball $B(e,j)$ for all $j \leq n$, we see that $\overline{\sigma}(u)(V_-)=V_-$. We immediately deduce that $\overline{\sigma}(u)(V_+)=V_+$ as well, as otherwise $\alpha_{n+1}$ would not be injective. As $\|u'x\| = n+1 > n = \|u\|$ we see that $x \in V_+$ and, consequently $\|\alpha_{n+1}(u)\| = n+1$. This then implies that $\|v\| = n+1$ as well, so both $u,v \in \partial B(e, n+1)$. Let us pick $v' \in B( e, n)$ and $y \in V\Gamma$ be such that $v = v'y$. By definition,
    \begin{displaymath}
        \alpha_{n+1}(u) = \alpha_{n+1}(u')\overline{\sigma}(u')(x) = \alpha_{n+1}(v')\overline{\sigma}(v')(y) = \alpha_{n+1}(v).
    \end{displaymath}
   The rest of the proof is analogous to the first part, showing that
   \begin{displaymath}
       \alpha_{n+1}^{-1} \colon B(e,n+1) \to B(e,n+1)
   \end{displaymath}
   is a well-defined function.
   
   We have constructed a sequence of partial maps
   $\alpha_{n} \colon B(e, n) \to B(e, n)$, for all $n \in \mathbb{N}$, such that $\alpha_{n}$ is a graph automorphism and $\alpha_n \restriction_{B(e, m)} = \alpha_{m}$ whenever $m  < n$. We now define a new map $\alpha \colon \C_\Gamma \to \C_\Gamma$ as
   \begin{displaymath}
       \alpha(v) = \alpha_n(v) \quad\mbox{ whenever }\quad v \in B( e,n).
   \end{displaymath}
   It immediately follows from what we proved so far that $\alpha \in \A_\Gamma$ and $\sigma(\alpha, v) = \overline{\sigma}(v)$ for all $v \in W_\Gamma$. Thus, $\overline{\sigma}$ is a legal configuration of local actions, and the proof is complete.
\end{proof}

Clearly, every $\alpha \in \A_\Gamma$ is fully determined by a pair $(w, \overline{\sigma}) \in W_\Gamma \times \Sigma_\Gamma$, where $w = \alpha(e)$ and $\overline{\sigma}(v) = \sigma(\alpha, v)$ for all $v \in W_\Gamma$. In particular, the group $\A_\Gamma$ is in bijection with the set $W_\Gamma \times \Sigma_\Gamma$ and, consequently, the stabiliser $\stab_{\A_\Gamma}(e)$ is in bijection with the set $\Sigma_\Gamma$.

\begin{lemma}
    Suppose that the the graph $\Gamma$ contains a good separating set $S \subseteq\Gamma$, and let $\Gamma_1$ be the corresponding subgraph such that there is a non-trivial $\nu \in \Aut(\Gamma_1)$ with $\nu\restriction_S = \id_S$. Then $\lvert \Sigma_\Gamma\rvert \geq 2^{\lvert W_\Gamma/W_{\Gamma_1}\rvert}$.
\end{lemma}
\begin{proof}
First, let $\nu' \in \Aut(\Gamma)$ denote the extension of $\nu$ to the whole of $\Gamma$, i.e.
\begin{displaymath}
    \nu'(x) = \begin{cases}
                            \nu(x) & \quad\mbox{if }x \in \Gamma_1;\\
                            x      & \quad\mbox{otherwise}.
                        \end{cases}
\end{displaymath}
Let $\chi \subseteq W_\Gamma/W_{\Gamma_1}$ be arbitrary. We define an element $\overline{\nu}_\chi \in \Aut(\Gamma)^{W_\Gamma}$ in a following way:
\begin{displaymath}
    \overline{\nu}_\chi(w) =    \begin{cases}
                                \nu'             &\quad\mbox{if $w W_{\Gamma_1} \in \chi$};\\
                                \id_{\Gamma}    &\quad\mbox{otherwise}.
                            \end{cases}
\end{displaymath}
We claim that $\overline{\nu}_\chi \in \Sigma_\Gamma$, that is  $\overline{\nu}_\chi$ is a legal configuration of local actions. Let $w \in W_\Gamma$ and $x \in V\Gamma$ be arbitrary. If $w W_{\Gamma_1} = wxW_{\Gamma_1}$ then, by definition,  
$\overline\nu_\chi(w) = \overline\nu_\chi(wx)$. Now, let us suppose that $w W_{\Gamma_1} \neq wxW_{\Gamma_1}$. This means that $x \in V\Gamma \setminus \Gamma_1$ and, in particular, $\star(x) \subseteq (V\Gamma \setminus \Gamma_1)\cup S$. We immediately see that $\overline{\nu}_\chi(w)\restriction_{\star(x)} = \id_\Gamma$ and, similarly, that $\overline{\nu}_\chi(wx)\restriction_{\star(x)} = \id_\Gamma$. Following Lemma \ref{lemma:legal configurations}, $\overline{\nu}_\chi$ is a legal configuration. This means that $\Sigma_\Gamma$ contains a subset of cardinality $2^{|W_\Gamma/W_{\Gamma_1}|}$, therefore
\begin{displaymath}
    2^{|W_{\Gamma}/W_{\Gamma_1}|} \leq |\Sigma_\Gamma|
\end{displaymath}
and the lemma is proved.
\end{proof}

As we discussed earlier, the stabiliser $\stab_{\A_\Gamma}(e)$ is in bijection with the set $\Sigma_\Gamma$ in a natural way. Also, one can verify that if $S \subset V\Gamma$ is a good separating set and $\Gamma_1$ is the corresponding factor, then the subgroup  
$W_{\Gamma_1}$ has infinite index in $W_\Gamma$.
Combining these with the simple observation that $\stab_{\A_\Gamma}(w) = w \stab_{\A_\Gamma}(e) w^{-1}$ we get:
\begin{corollary}
    \label{cor:uncountable stabilisers}
    Let $\Gamma$ be a finite weighted graph. If $\Gamma$ admits a good separating set, then the group $\A_\Gamma$ admits uncountable vertex stabilisers.
\end{corollary}

\begin{remark}
For an example when the separating set is not empty, we can consider the graph
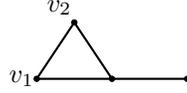
\begin{equation*}
\begin{tikzpicture}
\draw[thick] 
(1,0) --(0,0) --(.5,.75) --(1,0) --(2,0);
\draw[fill] 
(0,0) circle [radius=0.035]
node[xshift=-.2cm] {\small $v_1$}
(.5,.75) circle [radius=0.035]
node[xshift=-.2cm,yshift=.2cm] {\small $v_2$}
(1,0) circle [radius=0.035]
(2,0) circle [radius=0.035];
\end{tikzpicture}
\end{equation*}
and the automorphism $\alpha$ that interchanges the vertices $v_1$ and $v_2$ and fixes the other two vertices. Corollary~\ref{cor:uncountable stabilisers} guarantees that the automorphism group of the Cayley graph $\C_\Gamma$ is uncountable. If $v\in \C_\Gamma$ is a vertex, the automorphism $\alpha_v$ is given by reflection around the diagonal square
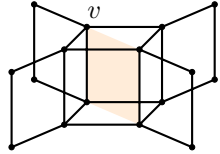
\begin{equation*}
\begin{tikzpicture}
\filldraw[draw=none, color=orange!15]
(-.2,-.2) --(-.2,.8) -- (.5,.5) --(.5,-.5) --(-.2,-.2);

\draw[thick] 
(.5,.5) --(.5,-.5) --(-.5,-.5) --(-.5,.5) --(.5,.5)

(-.2,-.2) -- (-.2,.8) -- (.8,.8) -- (.8,-.2) --(-.2,-.2)

(.5,.5) --(.8,.8)
(.5,-.5) --(.8,-.2)
(-.5,-.5) --(-.2,-.2)
(-.5,.5) -- (-.2,.8)

(-.5,-.5)--(-1.2,-.8) --(-1.2,.2) --(-.5,.5)

(.5,.5)--(1.2,.2) --(1.2,-.8) --(.5,-.5)

(.8,.8) --(1.5,1.1) -- (1.5,.1) --(.8,-.2)

(-.2,-.2) --(-.9,.1) --(-.9,1.1) -- (-.2,.8);

\draw[fill] (.5,.5) circle [radius=0.035] 
(-.5,.5) circle [radius=0.035] 
(-.5,-.5) circle [radius=0.035] 
(.5,-.5) circle [radius=0.035]
(-.2,-.2) circle [radius=0.035]
(-.2,.8) circle [radius=0.035]
node[xshift=.1cm,yshift=.2cm] {\small $v$}
(.8,.8) circle [radius=0.035]
(.8,-.2) circle [radius=0.035]

(-1.2,-.8) circle [radius=0.035]
(-1.2,.2) circle [radius=0.035]

(1.2,.2) circle [radius=0.035]
(1.2,-.8) circle [radius=0.035]

(1.5,1.1) circle [radius=0.035]
(1.5,.1) circle [radius=0.035]

(-.9,.1) circle [radius=0.035]
(-.9,1.1) circle [radius=0.035];
\end{tikzpicture}
\end{equation*}
Also, notice that Corollary \ref{cor:uncountable stabilisers} provides examples of one-ended right-angled Coxeter groups whose Cayley graph has uncountable automorphism group. For instance, one can consider the following graph.
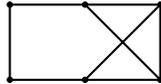
\begin{equation*}
\begin{tikzpicture}
\draw[thick]
(0,0.5) -- (0,-0.5)
(0,-0.5) -- (1,-0.5)
(1,0.5) -- (0,0.5)
(1,0.5) -- (2, 0.5)
(1,-0.5) -- (2,-0.5)
(2, 0.5) -- (2, -0.5)
(1, 0.5) -- (2, -0.5)
(1, -0.5) -- (2, 0.5);
\draw[fill] 
(0, 0.5) circle [radius=0.035]
(0,-0.5) circle [radius=0.035]
(1,-0.5) circle [radius=0.035] 
(1, 0.5) circle [radius=0.035]
(2,-0.5) circle [radius=0.035] 
(2, 0.5) circle [radius=0.035];
\end{tikzpicture}
\end{equation*}
\end{remark}

From Proposition \ref{prop_countable} we immediately deduce that Cayley graphs of right-angled Coxeter groups associated to  \emph{path graphs} $P_n$ and to \emph{cycle graphs} $C_n$ are countable.

A path graph, denoted $P_n$, is a connected graph with $n$ vertices, of which exactly two have valency one, whereas the remaining $n-2$ have valency two, and all edged have weight two.
A cycle graph, denoted $C_n$, is a connected graph with $n$ vertices of valency two and edges with weight two.

It is not true that if $\Delta$ is an induced subgraph of $\Gamma$ then $\Aut\bigl(\Cay(W_\Delta)\bigr)$ is a subgroup of $\A_\Gamma$. Consider the full graph $\Delta$ on two vertices, and
let $\Gamma=P_3$,  
so that $\Delta$ is an induced subgraph of $\Gamma$. It is easy to recognize that the group $\Aut \bigl(\Cay(W_\Delta)\bigr)$ is not a subgroup of $\Aut \bigl(\C_\Gamma\bigr)$. Indeed $W_\Delta\cong 
\mathbb{Z}_2\times \mathbb{Z}_2$ and $\Aut\bigl(\Cay(W_\Delta)\bigr)$ is the dihedral group $D_8$, the reflection group of a square. In particular, there is an element of order four in $\Aut\bigl(\Cay(W_\Delta)\bigr)$.
On the other hand, there is no element of order four in $\A_\Gamma$.

More is true. Indeed, from what we proved we immediately observe that there are graphs $\Gamma$ with induced subgraphs $\Delta$ such that $\A_\Gamma$ is countable, but $\Aut\bigl(\Cay(W_\Delta)\bigr)$ is uncountable. As an example, we can take $\Gamma=P_4$, and $\Delta$ to be the graph obtained from $\Gamma$ by removing a vertex of valency two.
In this case, the Cayley graph of $W_\Gamma$ is
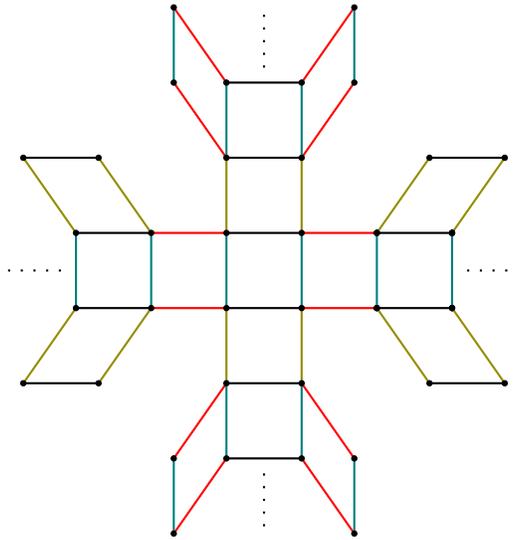
\begin{figure}[h]
\begin{tikzpicture}
\draw[thick, loosely dotted] 
(-2.2,0) --(-3,0)
(3.2,0) --(4,0)
(.5,2.7) --(.5,3.5)
(.5,-2.7) --(.5,-3.5)
;

\draw[thick,olive]
(0,.5) --(0,1.5)
(1,.5) --(1,1.5)
(0,-.5) --(0,-1.5)
(1,-.5) --(1,-1.5)
(-1,.5) --(-1.7,1.5)
(-2,.5) --(-2.7,1.5)
(-1,-.5) --(-1.7,-1.5)
(-2,-.5) --(-2.7,-1.5)
(2,.5) --(2.7,1.5)
(3,.5) --(3.7,1.5)
(2,-.5) --(2.7,-1.5)
(3,-.5) --(3.7,-1.5)
;

\draw[thick,teal]
\foreach \i in {-2,...,3} {
(\i,.5)--(\i,-.5)
}
(0,1.5) --(0,2.5)
(1,1.5) --(1,2.5)
(0,-1.5) --(0,-2.5)
(1,-1.5) --(1,-2.5)
(-.7,2.5)--(-.7,3.5)
(-.7,-2.5)--(-.7,-3.5)
(1.7,2.5)--(1.7,3.5)
(1.7,-2.5)--(1.7,-3.5)
;

\draw[thick]
\foreach \i in {-2,0,2} 
{
(\i,.5)--(\i+1,.5)
(\i,-.5)--(\i+1,-.5)
}
(0,1.5) --(1,1.5)
(0,2.5) --(1,2.5)
(0,-1.5) --(1,-1.5)
(0,-2.5) --(1,-2.5)
(-1.7,1.5) --(-2.7,1.5)
(-1.7,-1.5) --(-2.7,-1.5)
(2.7,1.5) --(3.7,1.5)
(2.7,-1.5) --(3.7,-1.5)
;

\draw[thick,red]
\foreach \i in {-1,1} {
(\i,.5)--(\i+1,.5)
(\i,-.5)--(\i+1,-.5)}
(0,1.5) --(-.7,2.5)
(0,2.5) --(-.7,3.5)
(1,1.5) --(1.7,2.5)
(1,2.5) --(1.7,3.5)
(0,-1.5) --(-.7,-2.5)
(0,-2.5) --(-.7,-3.5)
(1,-1.5) --(1.7,-2.5)
(1,-2.5) --(1.7,-3.5)
;

\filldraw 
(0,0.5) circle[radius=0.035] 
(1,0.5) circle[radius=0.035] 
(2,0.5) circle[radius=0.035] 
(3,0.5) circle[radius=0.035] 
(-1,0.5) circle[radius=0.035] 
(-2,0.5) circle[radius=0.035]

(0,-0.5) circle[radius=0.035] 
(1,-0.5) circle[radius=0.035] 
(2,-0.5) circle[radius=0.035] 
(3,-0.5) circle[radius=0.035] 
(-1,-0.5) circle[radius=0.035]
(-2,-0.5) circle[radius=0.035]

(0,1.5)  circle[radius=0.035]
(0,2.5)  circle[radius=0.035]
(1,1.5)  circle[radius=0.035]
(1,2.5) circle[radius=0.035]
(0,-1.5)  circle[radius=0.035]
(0,-2.5) circle[radius=0.035]
(1,-1.5)  circle[radius=0.035]
(1,-2.5) circle[radius=0.035]
(-.7,2.5) circle[radius=0.035]
(-.7,3.5) circle[radius=0.035]
(-.7,-2.5) circle[radius=0.035]
(-.7,-3.5) circle[radius=0.035]
(1.7,2.5) circle[radius=0.035]
(1.7,3.5) circle[radius=0.035]
(1.7,-2.5) circle[radius=0.035]
(1.7,-3.5) circle[radius=0.035]

(2,.5) circle[radius=0.035]
(2.7,1.5) circle[radius=0.035]
(3,.5) circle[radius=0.035]
(3.7,1.5) circle[radius=0.035]
(2,-.5) circle[radius=0.035]
(2.7,-1.5) circle[radius=0.035]
(3,-.5) circle[radius=0.035]
(3.7,-1.5) circle[radius=0.035]
(-1.7,1.5) circle[radius=0.035]
(-2.7,1.5) circle[radius=0.035]
(-1.7,-1.5) circle[radius=0.035]
(-2.7,-1.5) circle[radius=0.035]
;
\end{tikzpicture}
\caption{Cayley graph of $W_{\Gamma}$, for $\Gamma=P_4$.}
\label{fig_cayleygraph_P4}
\end{figure}
and has countable automorphism group by Proposition~\ref{prop_countable}.

This should be compared with the case of right-angled Artin groups \cite{T}, where this behaviour is not present.

\section{Beyond locally finite graphs}\label{section6}
Even though the main aim of this note was to classify those finitely generated Coxeter groups whose Cayley graph (with respect to the standard generating set) has a nondiscrete automorphism group, the tools we developed to do so allow us to construct explicit examples of vertex-transitive graphs of infinite degree that still have locally compact automorphism groups. In particular in Example \ref{weird example} gives construction of a graph with compact vertex stabilisers and Example \ref{weirder example} given an explicit construction of a graph such that edge stabilisers are compact-open but vertex stabilisers are not compact.

Before we proceed, let us recall some standard notation. Given a simplicial graph $\Gamma = (V\Gamma, E\Gamma)$, we will use $\Gamma^c$ to denote the \emph{complement graph} of $\Gamma$, i.e $V\Gamma^C = V\Gamma$ and $E\Gamma^c = \binom{V\Gamma}{2} \setminus E\Gamma$. Quite clearly, $\Aut(\Gamma^c)$ is naturally isomorphic to $\Aut(\Gamma)$ in an obvious way: $\phi \in \Sym(V\Gamma)$ is an automorphism of $\Gamma$ is and only if it is an automorphism of $\Gamma^c$.
\begin{lemma}
    \label{lemma:tree_complements}
    Suppose that $\Gamma$ is an infinite tree without leaves. Then $\Gamma^c$ does not admit a good separating set.  
\end{lemma}
\begin{proof}
    We will use Lemma \ref{char_good_sep_set}.
    
    Let $v \in V\Gamma$ and suppose that $\alpha \in \Aut(\Gamma)$ is a automorphism such that $\alpha \restriction_{\star_{\Gamma^c}(v)} = \id_\Gamma$. Following the definition of the complement graph, we see that $\star_{\Gamma^c}(v) = V\Gamma \setminus \star_\Gamma(v)$, therefore $\alpha$ must fix all vertices of the tree $\Gamma$ outside of $V\Gamma \setminus \star_\Gamma(v)$. Since $\Gamma$ does not have any leaves, we see that $v$ is adjacent to at least two distinct vertices, denote them $u_1, u_2$. By a similar argument, we see that $u_1$ must be adjacent to some vertex $u_1'$ outside of $\star_\Gamma(v)$ and $u_2$ must be adjacent to some $u_2'$ outside of $\star_\Gamma(v)$. By assumption, $u_1'$ and $u_2'$ are fixed by $\alpha$. Since $\Gamma$ is a tree, the sequence of vertices $(u_1', u_1, v, u_2, u_2')$ is the unique path between $u_1'$ and $u_2'$ and therefore the vertices $u_1, v, u_2$ must be fixed by $\alpha$ as well. Since $\Gamma$ does not have any leaves, we see that if $u \in \star_\Gamma(v)$ is a neighbour of $v$ it must be adjacent to some $u' \in V\Gamma \setminus \star_\Gamma(v)$. Again, $u'$ must be fixed by $\alpha$ by assumption and therefore, since $(v, u, u')$ is the unique path from $v$ to $u'$, we see that $u$ must be fixed as well. It then follows that $\alpha = \id_\Gamma$.

    Therefore, for every vertex $v \in V\Gamma$ and every $\alpha \in \Aut(\Gamma) \setminus \id_\Gamma$ we have that $\alpha \restriction_{\star_\Gamma(v)} \neq \id_\Gamma$, and thus, by Lemma \ref{char_good_sep_set}, $\Gamma$ does not admit a good separating set. 
\end{proof}

 The main idea behind the constructions presented in this section relies on the following application of Proposition \ref{prop_countable}
\begin{lemma}
\label{lemma:rigid_stabilisers}
    Suppose that $\Gamma$ is a simplicial graph which does not admit a good separating set. Then for any $v \in W_{\Gamma}$, the stabiliser $\stab_{\A}(v)$, where $\A = \Aut(\Cay(W_{\Gamma}, V\Gamma))$ is isomorphic to $\Aut(\Gamma)$ as a topological group.
\end{lemma}
\begin{proof}
    As the graph $\Gamma$ does not admit a good separating set, we see that $\A_\Gamma \simeq W_\Gamma \rtimes \Aut(\Gamma)$ by Proposition \ref{prop_countable}. In particular, we see that for any $v \in W_\Gamma$ we have that $\stab_{\A}(v)$ is isomorphic to $\Aut(\Gamma)$ as a group. It remains to show that the isomorphism is continuous.
    
    If the graph $\Gamma$ is finite, then both $\stab_{\A}(v)$ and $\Aut(\Gamma)$ are finite and therefore discrete. It follows trivially that $\stab_{\A}(v)$ and $\Aut(\Gamma)$ are isomorphic as topological groups as well.

    Let $v \in W_{\Gamma}$ be arbitrary and let $\phi \colon \Aut(\Gamma) \to \stab_{\A}(v)$ be the map sending $\sigma \in \Aut(\Gamma)$ to $\alpha_{\sigma}$, where $\alpha_{\sigma}$ is the unique element of $\stab_{\A}(v)$ such that $\sigma(\alpha_{\sigma}, u) = \sigma$ for all $u\in W_{\Gamma}$. Let $\mathcal{O} \subseteq \stab_{\A}(v)$ be open. Without loss of generality we may assume that there are $\beta \in \stab_{\A}(v)$ and a finite set $F \subseteq W_{\Gamma}$ such that
    \begin{displaymath}
        \mathcal{O} = \beta \stab_A(F) \cap \stab_{\A}(v).
    \end{displaymath}
    In fact, since the elements of $\stab_{\A}(v)$ are fully determined by their action on the neighbourhood of $v$, we may assume that $F \subseteq B(1, v)$. This means that there are $x_1, \dots, x_n \in V\Gamma$ such that $F = \{vx_1, \dots, vx_n\}$ and, clearly,
    \begin{displaymath}
        \phi^{-1}(\mathcal{O}) = \{ \sigma' \in \Aut(\Gamma) \mid \sigma'\restriction_{F'} = \sigma(\beta, v)\restriction_{F'}\},
    \end{displaymath}
    where $F' = \{x_1, \dots, x_n\} \subseteq V\Gamma$, so we see that $\phi^{-1}(\mathcal{O})$ is open in the permutation topology on $\Aut(\Gamma)$ and therefore $\phi$ is continuous. The proof that $\phi^{-1}$ is continuous is essentially the same and is left as an exercise.

    We see that $\phi \colon \Aut(\Gamma) \to \stab_{\A}(v)$ is both a group isomorphism and a homeomorphism of topological spaces, therefore $\Aut(\Gamma)$ and $\stab_{\A}(v)$ are isomorphic as topological groups.
\end{proof}

We now use Lemma \ref{lemma:tree_complements} and Lemma \ref{lemma:rigid_stabilisers} to give explicit constructions of vertex transitive graphs of infinite degree with locally compact group of automorphisms.
\begin{example}
    \label{weird example}
    Suppose that $\Gamma$ is an infinite regular rooted tree of finite degree. Then $\Cay(W_{\Gamma^c}, V\Gamma^c)$ is a vertex-transitive graph of infinite degree such that $\A = \Aut(\Cay(W_{\Gamma^c}, V\Gamma^c))$ is locally compact group with compact vertex-stabilisers. In particular, for any $w \in W_{\Gamma^c}$ we see that $\stab_{\A}(w)$ is isomorphic (as a topological group) to $\Aut(\mathcal{T}_d)$.
\end{example}
\begin{example}
    \label{weirder example}
    Suppose that $\Gamma = \mathcal{T}_d$ is an infinite $d$-regular tree. Then $\Cay(W_{\Gamma^c}, V\Gamma^c)$ is a vertex-transitive graph of infinite degree such that $\A = \Aut(\Cay(W_{\Gamma^c}, V\Gamma^c))$ is locally compact group such that edge stabilisers are compact but vertex stabilisers are not. In particular, for any $w \in W_{\Gamma^c}$ and $x \in V\Gamma^c$ we see that $\stab_{\A}(w)$ is isomorphic (as a topological group) to $\Aut(\mathcal{T}_d)$ and $\stab_{\A}(\{w,wx\})$ is isomorphic to a vertex stabiliser in $\Aut(\mathcal{T}_d)$, i.e. $\Aut(T_{d, d-1}^*)$, where $T_{d, d-1}^*$ is a infinite rooted tree such that the root has $d$ children and every other vertex has exactly $d-1$ children.
\end{example}


\bibliographystyle{plain}
\bibliography{references}

\begin{thebibliography}{1}

\bibitem{BB}
Anders Bj\"{o}rner and Francesco Brenti.
\newblock {\em Combinatorics of {C}oxeter groups}, volume 231 of {\em Graduate
  Texts in Mathematics}.
\newblock Springer, New York, 2005.

\bibitem{clarke}
Graham Clarke.
\newblock Automorphisms of geometric structures associated to coxeter groups,
  2011.

\bibitem{haglund}
Fr\'{e}d\'{e}ric Haglund and Fr\'{e}d\'{e}ric Paulin.
\newblock Simplicit\'{e} de groupes d'automorphismes d'espaces \`a courbure
  n\'{e}gative.
\newblock In {\em The {E}pstein birthday schrift}, volume~1 of {\em Geom.
  Topol. Monogr.}, pages 181--248. Geom. Topol. Publ., Coventry, 1998.

\bibitem{Halin}
R.~Halin.
\newblock Automorphisms and endomorphisms of infinite locally finite graphs.
\newblock {\em Abh. Math. Sem. Univ. Hamburg}, 39:251--283, 1973.

\bibitem{ImSm}
Wilfried Imrich and Simon~M. Smith.
\newblock On a theorem of {H}alin.
\newblock {\em Abh. Math. Semin. Univ. Hambg.}, 87(2):289--297, 2017.

\bibitem{LdlS20}
Paul-Henry Leemann and Mikael de~la Salle.
\newblock Cayley graphs with few automorphisms.
\newblock {\em J. Algebraic Combin.}, 53(4):1117--1146, 2021.

\bibitem{LdlS}
Paul-Henry Leemann and Mikael de~la Salle.
\newblock Cayley graphs with few automorphisms: the case of infinite groups.
\newblock {\em Ann. H. Lebesgue}, 5:73--92, 2022.

\bibitem{T}
Thomas Taylor.
\newblock Automorphisms of cayley graphs of right-angled artin groups.
\newblock Honours thesis, University of Newcastle, 2017.

\end{thebibliography}

\end{document}